\documentclass[a4paper,11pt]{article}
\usepackage{makeidx}
\usepackage[french,english]{babel}
\usepackage[latin1]{inputenc}
\usepackage{amsfonts}
\usepackage{bm}
\usepackage{amssymb}
\usepackage{latexsym}
\usepackage{amsmath}
\usepackage{amscd}
\usepackage{amsthm}
\usepackage{hyperref}
\usepackage{rotating}
\usepackage{graphicx}
\usepackage{pifont}
\usepackage{curves}
\usepackage{fancyhdr}
\usepackage{epsfig}
\usepackage{pstricks}
\usepackage{pst-tree}
\usepackage{subfigure}
\usepackage{epic}
\usepackage{alltt}
\usepackage[all]{xy}
\usepackage{appendix}
\usepackage{ulem}

\theoremstyle{plain}
\newtheorem{theorem}{Theorem}[section]
\newtheorem{prop}[theorem]{Proposition}
\newtheorem{lemme}[theorem]{Lemma}

\theoremstyle{definition}

\newtheorem{hyp}[theorem]{Assumption}
\newtheorem{ex}[theorem]{Example}
\newtheorem{rque}[theorem]{Remark}

\title{Slow and fast scales for superprocess limits of age-structured populations}

\author{Sylvie M\'{e}l\'{e}ard\footnote{CMAP, Ecole Polytechnique} and Viet Chi Tran\footnote{Equipe Probabilit\'{e} Statistique, Laboratoire Paul Painlev\'{e}, USTL}}

\date{\today}

\numberwithin{equation}{section}

\newcommand{\Co}{\mathcal{C}}

\def\D{\mathbb{D}}

\def\N{\mathbb{N}}

\def\P{\mathbb{P}}
\def\R{\mathbb{R}}
\def\E{\mathbb{E}}

\def\M{\mathcal{M}}

\def\X{\mathcal{X}}

\def\ind{{\mathchoice {\rm 1\mskip-4mu l} {\rm 1\mskip-4mu l}
{\rm 1\mskip-4.5mu l} {\rm 1\mskip-5mu l}}}

\def\eg{\textit{e.g.} }
\def\ie{\textit{i.e.} }

\def\etal{\textit{et al.} }

\newcommand{\MF}{\mathcal{M}_F}

\newcommand{\me}{\medskip \noindent}
\newcommand{\bi}{\bigskip \noindent}
\newcommand{\be}{\begin{eqnarray}}
\newcommand{\ee}{\end{eqnarray}}
\newcommand{\ben}{\begin{eqnarray*}}
\newcommand{\een}{\end{eqnarray*}}

\begin{document}

\maketitle

\begin{abstract}
A superprocess limit for an interacting birth-death particle system modelling a population with trait and physical age-structures is established. Traits of newborn offspring are inherited from the parents except when mutations occur, while ages are set to zero. Because of interactions between individuals, standard approaches based on the Laplace transform do not hold. We use a martingale problem approach and a separation of the slow (trait) and fast (age) scales. While the trait marginals converge in a pathwise sense to a superprocess, the age distributions, on another time scale, average to equilibria that depend on traits. The convergence of the whole process depending on trait and age,  only holds  for finite-dimensional time-marginals.
We apply our results to the study of  examples illustrating different cases of trade-off between competition and senescence.
\end{abstract}


\section{Introduction}

We consider an asexual population in which the survival probability and reproduction rate of each individual are characterized by a quantitative trait,
such as for example the body size, or the rate of food
intake. As emphasized by Charlesworth \cite{charlesworth}, most of these abilities also depend on age. In this paper, we are interested in studying the joint effects of age and trait structures in the interplay between ecology and evolution. Evolution, acting on the distribution of traits in the population,
is the consequence of three basic mechanisms: heredity, which transmits traits to new offspring, mutation, which drives the
variation in the trait values, and selection
between these different trait values, which is due to ecological interactions. Some questions on evolution are strongly related to the age structure. For example, we would like to understand how the age influences the persistence of the population or the trait's evolution, or which age structure will appear in long time scales for a given trait. 

Our model relies on an individual-based birth and death process with age and trait introduced in  M\'{e}l\'{e}ard and Tran \cite{meleardtran} (see also Ferri\`{e}re and Tran \cite{ferrieretran}). It generalizes the trait-structured case developed in Champagnat \etal  \cite{champagnatferrieremeleard,champagnatferrieremeleard2} and the age-structured case in Jagers and Klebaner \cite{jagersklebaner,jagersklebaner2}, Tran \cite{trangdesdev}. Here, each individual is characterized by a quantitative trait and by its physical age, that is the time since its birth. We describe the dynamics of large populations composed of small individuals with short lives. Life-lengths and durations between reproductions are assumed to be proportional to the individuals' weights, and these weights are inversely proportional to the population size. At each birth, the age is reset to $0$, inducing a large asymmetry between the mother and her daughters. When a mutation occurs, the new mutant trait is close to its ancestor's one, yielding a slow variation of the trait. Hence, these two mechanisms lead to a difference of time scales between age and trait. Moreover, we take resource constraints into account by including competition between individuals.

Our aim is to study an approximation of this model when the population size tends to infinity. Our main result shows that two qualitatively different asymptotic behaviors arise from the separation of the fast age and slow trait time scales.
While the trait marginals converge in a pathwise sense to a superprocess, the age distributions stabilize into deterministic equilibria that depend on the traits.
To our knowledge, nothing has been done before in the setting we are interested in, with slow-fast variables and nonlinearity.
The techniques we use are based on martingale properties and generalize to this infinite dimensional setting the  treatment of the slow-fast scales for diffusion processes, developed by  Kurtz \cite{kurtzaveraging}, and by Ball \etal \cite{ballkurtzpopovicrempala}.

Our results generalize Athreya \etal \cite{athreyaathreyaiyer}, Bose and Kaj \cite{bosekaj}, where averaging phenomena are proved in the case where birth and death rates do not depend on age. In our case with dependence, the lifelength of an individual cannot be governed by an age distribution function independent of trait, or by a positive continuous additive functional, as in Bose and Kaj \cite{bosekaj2}, Dawson \etal \cite{dawsongorostizali}, Dynkin \cite{dynkin91}, Kaj and Sagitov \cite{kajsagitov}, Fleischmann \etal \cite{fleischmannvatutinwakolbinger}, Wang \cite{wangbio}. Moreover the Laplace characterization that these authors use extensively does not hold anymore when interactions between particles are allowed.
In Evans and Steinsaltz  \cite{evanssteinsaltz}, the damage segregation at cell fissioning is considered as an age. Nevertheless in their model there is no interaction between cells, and at each birth, the daughters' ages (as damages) are not reset to zero, but distributed asymmetrically following a distribution centered on the mother's age.\\

In Section \ref{sectionmicroscopique}, the population dynamics is described by an individual-based point measure-valued birth and death process with an age and trait structures. In Section \ref{sectionsuperprocessus}, we establish our main limit theorem (Theorem \ref{propconvergence}) where the averaging phenomenon is obtained. In Section \ref{sectionexemples}, we consider as an illustration two models where the population is structured by size and physical or biological age.  We present simulations and comment the different behaviors.
\\


\noindent \textbf{Notation:}
For a given metric space $E$, we denote by $\mathbb{D}([0,T],E)$ the space of right continuous and left limited (c\`ad-l\`ag) functions from $[0,T]$ to $E$. This space is embedded with the Skorohod topology (\eg \cite{rachev, joffemetivier}).

\par If $\X$ is a subset of $\R^d$, we denote by $\mathcal{M}_F(\X)$ the set of finite measures on $\X$, which will be  usually embedded with the topology of weak convergence. Nevertheless, if $\X$ is unbounded, we will also consider  the topology of vague convergence. If we need to differentiate both topological spaces, we will denote by  $(\mathcal{M}_F(\X),w)$, respectively  $(\mathcal{M}_F(\X),v)$, the space of measures endowed by the weak (resp. vague) topology.  For a measurable real bounded function $f$, and a measure $\mu \in \mathcal{M}_F(\X)$, we will denote
$$\langle \mu,f\rangle = \int_{\X} f(x) \mu(dx).$$
For $\ell\in \N$, we denote by $\Co^\ell_b(\X,\R)$ the space of real bounded functions $f$ of class $\Co^\ell$ with bounded derivatives.
In the sequel, the space $\Co^{0,1}_b(\X\times \R_+,\R)$ (resp. $\Co_c(\X,\R)$, $\Co^1_c(\R_+,\R)$) denotes the space of continuous bounded real functions $\varphi(x,a)$ on $\X\times \R_+$ of class $\Co^1$ with respect to $a$ with bounded derivatives (resp. of continuous real functions on $\X$ with compact support, of $\Co^1$ real functions on $\R_+$ with compact support in $[0,+\infty)$).

\section{Microscopic age and trait structured particle system}\label{sectionmicroscopique}


We consider a discrete population in continuous time where the individuals reproduce, age and die with rates depending on a
hereditary trait and on their age. An individual is  characterized by a quantitative trait $x\in \X$ where $\X$ is a closed subset of $\R^d$ and by its physical age $a\in \R_+$, \ie the time since its birth. The individuals reproduce asexually during their lives, and the trait from the parent is transmitted to its offspring except when a mutation occurs.
Resources are shared by the individuals, implying an interaction described by a kernel comparing the competitors' traits and ages. Senescence, which quantifies the decrease of fertility or survival probability with age, is also taken into account. These two phenomena create selection pressure.\\

\noindent We are interested in approximating the dynamics of  a  large population whose size is parametrized by some integer $n$. This parameter can be seen as the order of the carrying capacity, when the total amount of resources is assumed to be fixed. If the parameter $n$ is large, there will be many individuals with little per capita resource and we renormalize the individual biomass by the weight $1/n$.

\noindent We consider here allometric demographies where the lifetime and  gestation length of each individual are proportional to its biomass. Thus the birth and death rates are  of  order  $n$, while preserving the demographic balance.
As a consequence the right scale to observe a nontrivial limit in the age structure, as $n$ increases, is of order $1/n$.
\\

\noindent The population at time $t$ is represented by a point measure as follows:
\begin{equation}
X^n_t = {1\over n}  \sum_{i=1}^{N^n_t}\delta_{(X_i(t),A_i(t))},\label{defht}
\end{equation}where $N^n_t=\langle n X^n_t,1\rangle$ is the number of individuals alive at time $t$, and $X_i(t)$ and $A_i(t)$ denote respectively the trait and age of individual $i$ at time $t$ (individuals are ranked in lexicographical order for instance).

\medskip \noindent
The dynamics of  $X^n$  is given as follows:
\begin{itemize}

\item The birth of an individual with trait $x\in \X$ and age $a\in \R_+$ is given by
$\ n\, r(x,a)+b(x,a)$.  The new offspring is  of age $0$ at birth.
Moreover, it inherits of the trait $x$ of its ancestor with probability $1-p(x,a)\in [0,1]$ and  is a mutant with probability $p(x,a)\in [0,1]$. The mutant trait  is then  $x+h$, where the variation $h$ is randomly chosen following  the distribution $\pi^n(x,dh)$.

 \item Individuals age with velocity $n$, so that the physical age at time $t$ of an individual born at time $c$ is $a=n(t-c)$.
\item The intrinsic death rate of an individual with trait $x\in \X$ and age $a\in \R_+$ is given by $\ n\, r(x, a)+d(x,a)$.  The competition between individuals $(x,a)$ and $(y,\alpha)$ is described by  the value $U((x,a), (y,\alpha))$ of a kernel $U$.  In a population described by the measure $X\in \mathcal{M}_F(\X\times \R_+)$, the total interaction on an individual $(x,a)$ is thus:
\begin{equation}
XU(x,a)=\int_{\X\times \R_+}U((x,a),(y,\alpha))X(dy,d\alpha),
\end{equation}
and its total death rate is
 $\ n\, r(x,a)+d(x,a)+XU(x,a)$.
\end{itemize}

\begin{hyp}\label{hyptaux}
\begin{enumerate}
\item The birth and death rates $b$ and $d$ are continuous on $ \X\times \R_+$ and  bounded respectively  by $\bar{b}$ 	and $\bar{d}$.

\item The function $r$ is continuous on  $\X\times \R_+$. There exist a positive constant $\bar{r}$ and a non-negative real function $\underline{r}$ such that  $\ \forall (x,a)\in \X\times \R_+,\  \underline{r}(a)\leq |r(x,a)|\leq \bar{r}$ with
    \begin{equation}
    \int_0^{+\infty}\underline{r}(a)da=+\infty.\label{conditionminorationr}
    \end{equation}

\item The competition kernel $U$ is continuous on $(\X\times \R_+)^2$ and  is bounded by $\bar{U}$.\hfill $\Box$
\end{enumerate}
\end{hyp}

Assumption \ref{hyptaux}-(2) implies that any individual $i$ from the population $X^n_t$ has a finite lifetime that is stochastically upper-bounded by a random variable $D^n_i(t)$ with survival function
\begin{equation}
S^n(\ell)=\P(D^n_i(t)>\ell)=\exp\Big(-\int_0^\ell n\, \underline{r}(nu)du\Big),\label{survie}
\end{equation}where we recall that the aging velocity is equal to $n$.\\
 If  the competition kernel  $U$ is positive on $(\X\times \R_+)^2$, it can model a competition  of the logistic type: the more important the size of the population is and the higher the death rate by competition is. For examples of such kernels we refer to \cite{meleardtran}.

\begin{ex}Let us illustrate the condition \eqref{conditionminorationr}.
\begin{enumerate}
\item If the function $r$ is lower bounded by a positive constant $\underline{r}$, then \eqref{conditionminorationr} is satisfied and so is \eqref{survie}, with exponential random variable $D^n_i(t)$ of parameter $n\underline{r}$.
\item Another example is when the trait $x$ is linked to the rate of metabolism, which measures the energy expended by individuals. Ageing may result from toxic by-products of the metabolism and we can define a biological age, $xa$. If $x\in [x_1,x_2]$ with $x_1, x_2>0$ and if we define $r(x,a)=xa$, then Condition \eqref{conditionminorationr} is satisfied with $\underline{r}(a)=x_1a$. The example of Section \ref{sectionex2} deals with this case.
\item If we consider $r(x,a)=\gamma/(1+a)$ with $\gamma \in (0,1)$, then \eqref{conditionminorationr} is also satisfied and the probability of observing an age higher than $a$ is equivalent to $a^{-\gamma}$ when $a$ tends to infinity. Such cases with distributions in the domain of attraction of a stable law, but without interaction, have been considered for instance in \cite{fleischmannvatutinwakolbinger}.
\end{enumerate}
\end{ex}

\begin{hyp}\label{tauxmutation}
For any $x\in\X$, the mutation kernel $\pi^n(x,dh)$ has its support in $\X-\{x\}=\{h\in \R^d \,|\,x+h\in \X\}$.
We consider two cases:

\begin{enumerate}
\item The trait space $\X$ is a compact subset of $\mathbb{R}^d$
 and  there exists a generator $A$ of a Feller semi-group on $\Co_b(\X,\R)$ with domain $\mathcal{D}(A)$ dense in $\Co_b(\X,\R)$ such that:
\begin{equation}
\forall f\in \mathcal{D}(A),\, \quad \lim_{n\rightarrow +\infty} \sup_{ x\in \X} \left|n\int_{\X-\{x\}} \big(f(x+h)-f(x)\big)\pi^n(x,dh) - Af(x)\right| =0.\label{condnoyaumutation}
\end{equation}

\item The trait space  $\X$ is a closed subset of $\mathbb{R}^d$ and we assume in addition that
  there exists $\ell_1\geq \ell_0\geq 2$ with $\Co_b^{\ell_1}(\X,\R) \subset   \mathcal{D}(A)$  and $\forall f\in \Co_b^{\ell_1}(\X,\R)$, $\forall x\in {\cal X}$.
\begin{equation} \label{estA}
|Af(x)| \leq  C \sum_{\substack{|k|\leq \ell_0 \\
k=(k_1,\dots,k_d)}}|D^k f (x)|\end{equation}
and
 \begin{equation}
 \sup_{ x\in \X} \left|n\int_{\X-\{x\}} \big(f(x+h)-f(x)\big)\pi^n(x,dh) - Af(x)\right| \leq \varepsilon_n \sum_{\substack{|k|\leq \ell_1 \\
k=(k_1,\dots,k_d)}}\|D^k f\|_\infty ,\label{estnoyaumutation}
\end{equation}
where $ D^{k}f(x)=\partial_{x_1}^{k_1}\dots\partial_{x_d}^{k_d}f(x)$,
 $\varepsilon_n$ is a sequence tending to $0$ as $n$ tends to infinity and $C$ is a constant.
\end{enumerate}
\end{hyp}

\begin{rque}
Both Assumptions \eqref{condnoyaumutation}
 and \eqref{estnoyaumutation} describe small mutation steps.
 The stronger hypothesis \eqref{estnoyaumutation} is required when  $\X$ is not compact, to obtain the tightness in the proof of Theorem \ref {propconvergence}.
 \end{rque}

\begin{ex}\label{exemple_noyau}Let us give some examples of mutation kernels satisfying  (\ref{condnoyaumutation}) or \eqref{estA} and \eqref{estnoyaumutation}.

\me
1.  In the case where $\X= [x_1, x_2]$,  the mutation kernel  $\pi^n(x,dh)$ is a Gaussian distribution with mean $0$ and variance $\sigma^2/n$, conditioned to $ [x_1-x, x_2 -x]$. In this case, elementary computation shows that for $f\in \Co^2_b([x_1,x_2],\R)$ such that $f'(x_1)= f'(x_2)=0$,
$Af(x)={\sigma^2\over 2} f''(x),$
which satisfies (\ref{condnoyaumutation}).

\me
2. In the case where $\X=\R^d$,  a possible choice of  mutation kernel $\pi^n(x,dh)$ is a Gaussian distribution with mean $0$ and covariance matrix $\Sigma(x)/n$, with $\Sigma(x) = (\Sigma_{ij}(x), 1\leq i,j\leq d)$. The generator $A$ is given for $f\in \Co^2_b(\mathbb{R}^d,\R)$  by $Af(x)={1\over 2} \sum_{i,j=1}^d \Sigma_{ij}(x) \partial_{ij}^2 f(x).$
If the function $\Sigma$ is bounded, then Assumption \eqref{estA} is fulfilled. If moreover, the third moments of  $\pi^n(x,dh)$ are bounded (in $x$), then  \eqref{estnoyaumutation} is satisfied.\hfill $\Box$
\end{ex}

\medskip

\noindent Let us now describe the generator  $\ L^n\ $ of the
$\mathcal{M}_F(\X\times \R_+)$-valued Markov process $X^n$, which  sums the aging phenomenon and  the ecological
dynamics of the population.
  As developed in Dawson \cite{dawson}
Theorem 3.2.6, the set of cylindrical functions defined for each
$\mu\in \mathcal{M}_F(\X\times \R_+)$ by $F_\varphi(\mu)=F(\langle
\mu,\varphi\rangle)$, with $F\in \Co^1_b(\R,\R)$ and $\varphi\in
\Co^{0,1}_b(\X\times \R_+,\R)$, generates the set of bounded
measurable functions on $\mathcal{M}_F(\X\times \R_+)$. For such function, for $\mu \in  \mathcal{M}_F(\X\times \R_+)$,
\begin{multline}
  L^nF_\varphi(\mu)=
n   \langle \mu,  \partial_a \varphi(.)\rangle F'_\varphi(\mu)\\
+
n \int_{\X\times \R_+} \left[\big(nr(x,a)+d(x,a)+\mu U(x,a)\big)
  \Big(F_\varphi\big(\mu-\frac{1}{n}\delta_{(x,a)}\big)-F_\varphi(\mu)\Big)\right.\\
  + \left. \big(nr(x,a)+b(x,a)\big)\int_{\R^d}
  \Big(F_\varphi\big(\mu+\frac{1}{n}\delta_{(x+h,0)}\big)-F_\varphi(\mu)\Big)K^n(x,a, dh)\right]\mu(dx,da),
\end{multline}
where \begin{equation}
K^n(x, a, dh)=p(x,a)\ \pi^n(x,dh) +(1- p(x,a))\ \delta_0(dh).\label{def_K^n}
\end{equation}

 \me  Under the condition $\sup_{n\in \N^*}\mathbb{E}\left(\langle X^n_0,1\rangle\right)<+\infty$, it has been proved  in M\'el\'eard-Tran \cite{meleardtran}  (see also the case without age in Champagnat-Ferri\`ere-M\'{e}l\'{e}ard \cite{champagnatferrieremeleard2} and the case without trait in Tran \cite{trangdesdev}), that there exists for any $n$,  a c\`ad-l\`ag Markov process with generator $L^n$, which  can be obtained as solution of a stochastic differential equation driven by a Point Poisson measure. Trajectorial uniqueness also holds for this equation. The construction provides an  exact individual-based simulation algorithm (see \cite{ferrieretran}).

 \medskip
 \noindent
 A slight adaptation of the proofs in  \cite{champagnatferrieremeleard2}  allows us to get the
\begin{prop}\label{prop_usefulresults} (i) Under Assumptions \ref{hyptaux}, and if
\begin{equation}\label{moment}\sup_{n\in \N^*}\E(\langle X^n_0,1\rangle^{3})<+\infty,\end{equation}
 then for all $T>0$,
\begin{align}
\sup_{n\in \N^*}\sup_{t\in [0,T]}\E\Big(\langle X^n_t,1\rangle^3\Big)<+\infty \mbox{ and } \sup_{n\in \N^*}\E\Big(\sup_{t\in [0,T]}\langle X^n_t,1\rangle^2\Big)<+\infty.\label{estimee_momentp}
\end{align}
(ii) Moreover, for $n\in \N^*$ and a test function $
\varphi\in \Co^{0,1}_b(\X\times \R_+,\R)$, the process $M^{n,\varphi}$ defined by
\begin{multline}
M^{n,\varphi}_t=\langle X^n_t,\varphi\rangle-  \langle X^n_0,\varphi\rangle -
n\, \int_0^t \langle X^n_s,\partial_a \varphi(x,a)\rangle\, ds\\
- \int_0^t \int_{\X\times \R_+} \Big(\big(nr(x,a)+b(x,a)\big)\int_{\R^d} \varphi(x+h,0) K^n(x, a, dh)\\
-\big(nr(x,a)+d(x,a)+X^n_s U(x,a)\big)\varphi(x,a)\Big)
X^n_s(dx,da)\,ds\label{PBM}\end{multline}
is a square integrable martingale started at 0 with quadratic variation:
\begin{multline}
\langle M^{n,\varphi}\rangle_t=  \frac{1}{n}\int_0^t \int_{\X\times \R_+} \Big(
\big(nr(x,a)+b(x,a)\big)\int_{\R^d} \varphi^2(x+h,0) K^n(x, a, dh)\\
+ \big(nr(x,a)+ d(x,a)+X^n_sU(x,a)\big)\varphi^2(x,a)\Big)
X^n_s(dx,da)\,ds.\label{crochetPBM}
\end{multline}
\end{prop}

\noindent Notice that in (\ref{PBM}), the mutation rate is hidden in the kernel $K^n$:
 \begin{multline}  \big(nr(x,a)+b(x,a)\big)\int_{\R^d} \varphi(x+h,0) K^n(x, a, dh)=  \big(nr(x,a)+b(x,a)\big)(1-p(x,a))  \varphi(x,0)\\
 + \big(nr(x,a)+b(x,a)\big) p(x,a) \int_{\R^d} \varphi(x+h,0) \pi^n(x,dh).
 \end{multline}



\section{Superprocess limit}\label{sectionsuperprocessus}

We now investigate the limit when $n$ increases to $+\infty$. In the limit, we obtain a continuum of individuals in which the individualities are lost. It is in particular difficult to keep track of the age-distribution when $n$ tends to infinity. Because of the non-local branching (a mother of age $a>0$ gives birth to a daughter of age $0$) and because the aging velocity tends to infinity, it is impossible to obtain directly the uniform tightness on $\D(\R_+,\MF(\X\times \R_+))$ of the sequence of laws of the measure-valued processes $(X^n_.(dx,da))_{n\in \N^*}$, as it can be observed considering \eqref{PBM}. Indeed, assuming that the function $\varphi$ only depends on $a$, the term of the form
 $$ \int_0^t \int_{\X\times \R_+} nr(x,a)\, (\varphi(0)-\varphi(a))\,
X^n_s(dx,da)ds$$
cannot be tight if $r(x,a)$ is bounded and $X^n$ is tight.
Therefore, we will be led to firstly show the uniform tightness of the trait marginal of the process $X^n$ and then to prove that in the limit,
 an averaging phenomenon  appears for the age dynamics. Indeed,  this  "fast" evolving component stabilizes in an equilibrium that depends on the dynamics of the "slow" trait component.

We generalize  to measure-valued processes, averaging techniques of Ball \etal \cite{ballkurtzpopovicrempala}, Kurtz \cite{kurtzaveraging} for diffusion processes.  A specificity in our case is that the fast-scaling is related to  time, since age is involved.
 In addition, notice that the competition between individuals creates a large dependence between the age and  trait distributions. At our knowledge this dependence has never been investigated before in the literature.

\bi
 Let us introduce the marginal $\bar{X}_t^n(dx) $ of $X^n_t(dx,da)$  defined for any bounded and measurable function $f$ on $\X$ and for any $t\in \R_+$ by
\begin{equation}
\int_{\X} f(x)\bar{X}_t^n(dx)=\int_{\X\times \R_+} f(x)X_t^n(dx,da).
\end{equation}

\noindent
  Our main result states the convergence of the sequence $(\bar{X}^n)_{n\in \N^*}$ to a nonlinear super-process. The nonlinearity remains at the slow time scale in  the growth rate, which is preserved in this asymptotics. Moreover,  fast mutations are compensated by  small mutation steps. Fast births and deaths   provide stochastic fluctuations in the limit.

\begin{theorem}\label{propconvergence} Assume Hypotheses \ref{hyptaux} and \ref{tauxmutation}, (\ref{moment}) and assume that there exists $ X_0\in \mathcal{M}_F(\X\times \R_+)$ such that $\, \lim_{n\rightarrow +\infty} X^n_0 = X_0$ in $(\mathcal{M}_F(\X\times \R_+),w)$, the limit being in probability for the sake of simplicity.

For any $x\in \X$, let us introduce the age probability density
\begin{align}
\widehat{m}(x,a)=\frac{\exp\big(-\int_0^a r(x,\alpha)d\alpha\big)}{\int_0^{+\infty}\exp\big(-\int_0^a r(x,\alpha)d\alpha\big)da}.\label{mchapeau}
\end{align}

\noindent Then, for each $T>0$,  the sequence $(\bar{X}^n)_{n\in\N^*}$ converges in law in $\D([0,T],(\mathcal{M}_F(\X),w))$
to the unique superprocess $\bar{X}\in \Co([0,T],\mathcal{M}_F(\X))$ such that for any function $f \in \mathcal{D}(A)$,\\
\begin{align}
M^{f}_t=  \langle \bar{X}_t,f\rangle -\langle \bar{X}_0,f\rangle
-&\int_0^t \int_{\X} \Big(
 \widehat{(p\,r)}(x)Af(x)\nonumber\\
+&\big[\widehat{b}(x)-\big(\widehat{d}(x)+\bar{X}_s\widehat{U}(x)\big)\big]f(x)\Big)\bar{X}_s(dx)\, ds\label{martingalepremiercas}
\end{align}is a square integrable martingale with quadratic variation:
\begin{align}
\langle M^f\rangle_t= \int_0^t \int_{\X}2\widehat{r}(x)f^2(x)\bar{X}_s(dx)\, ds.\label{crochetmartingalepremiercas}
\end{align}
Here, any $\widehat{\psi}(x)$ is defined for a bounded function $\psi(x,a)$ by  \begin{align*}
\widehat{\psi}(x)= & \int_{\R_+}\psi(x,a)\widehat{m}(x,a)da,\end{align*}
and $\bar{X}_t \widehat{U}(x)$ is given by
$$ \bar{X}_t\widehat{U}(x)= \int_{\X}\left(\int_{\R_+}\int_{\R_+} U((x,a),(y,\alpha))\widehat{m}(y,\alpha)d\alpha\ \widehat{m}(x,a)da\right)\bar{X}_t(dy).$$
\hfill $\Box$
\end{theorem}

\noindent
Theorem \ref{propconvergence} states that in the limit,
 an averaging phenomenon happens and the  "fast" age component finally submits to the dynamics of the "slow" trait component.  Since the fast-scaling (involving age) is   related to the time, the stable age distribution $\widehat{m}(x,a)da$ given in \eqref{mchapeau}
 is obtained  for each trait $x$ as the  long time limit in the age-structured population where all coefficients except $\ r(x,a)$ are zero.


\me
Before proving this slow-fast limit theorem, let us insist on the main difficulty created by the competition mechanism. Indeed,   the branching property fails and  it impedes the  use of Laplace-transform techniques, as it had almost systematically been done in the past papers studying  particle pictures with age-structures.
Our model  generalizes the age-structure population process studied in Athreya \etal \cite{athreyaathreyaiyer}, Bose and Kaj \cite{bosekaj}, in which  birth and death rates are equal to a constant $\lambda$.  In that case, the limiting  behaviour  of the renormalized critical birth and death process appears as  a particular case of Theorem  \ref{propconvergence}  with $\widehat{m}(x,a)da=\lambda e^{-\lambda a}da$. In  Dawson \etal \cite{dawsongorostizali},  Dynkin \cite{dynkin91}, and   Kaj and Sagitov \cite{kajsagitov}, the age dependence  is modelled through an additive functional of the motion
process. In that way, the age "accumulates" along the lineage. 
In our case, the age is set to zero at each birth, inducing a renewal phenomenon. The  life-length does not have a fixed probability distribution anymore, unless there is no interaction.
In \cite{bosekaj2}, the authors consider a  particle system with a different scaling, which favors large reproduction events. The limit in this case is not a superprocess anymore but  behaves as the solution of a McKendrick-Forster equation perturbed by random immigration events created by the large rare birth events.

\bi
The proof of Theorem \ref{propconvergence} is the aim  of  Section \ref{sectionsuperprocessus}. We firstly establish the tightness of the sequence
$ (\bar{X}^n)_{n\in \N^*}$ (Section \ref{sectiontension}).
To identify its limiting values $\bar X$, we need an intermediary step. We consider the measures $(X^n_{t}(dx,da)dt)_{n\in \N^*}$ and show that their limiting values are  equal to $\bar X_{t}(dx)\widehat{m}(x,a)da dt$. This implies that  $\bar X$ is  solution of the martingale problem  given in \eqref{martingalepremiercas}, \eqref{crochetmartingalepremiercas}.  Uniqueness in this martingale problem allows us to deduce the convergence of $ (\bar{X}^n)_{n\in \N^*}$. \\

\subsection{Tightness of $(\bar{X}^n)_{n\in \N^*}$}\label{sectiontension}

In this subsection, we shall prove that:

\begin{prop}\label{proptightness} Assume Hypotheses \ref{hyptaux} and \ref{tauxmutation}, and (\ref{moment}). If the sequence of laws of $(\bar{X}^n_0)_{n\in \N^*}$ is uniformly tight in $(\mathcal{M}_F(\X),w)$, then the sequence $(\mathcal{L}(\bar{X}^n))_{n\in \N^*}$ is uniformly tight in the space of probability measures on $\D([0,T],(\mathcal{M}_F(\X),w))$.
\end{prop}

\begin{proof}
Recall firstly that for a measurable and bounded function $f$ on $\X$, the process
\begin{align}
M^{n,f}_t&=\langle X^n_t,f\rangle-  \langle X^n_0,f\rangle
- \int_0^t \int_{\X\times \R_+} \Big(\big(nr(x,a)+b(x,a)\big)\int_{\R^d} f(x+h) K^n(x, a, dh)\nonumber\\
&\hskip1cm -\big(nr(x,a)+d(x,a)+X^n_s U(x,a)\big)f(x)\Big)
X^n_s(dx,da)\,ds\nonumber\\
&=\langle X^n_t,f\rangle-  \langle X^n_0,f\rangle
- \int_0^t \int_{\X\times \R_+} \Big(\big(b(x,a) -d(x,a)-X^n_s U(x,a)\big)\, f(x)
 \nonumber\\
&\hskip1cm + \big(nr(x,a)+b(x,a)\big) p(x,a)\int_{\R^d} \big(f(x+h)-f(x)\big) \pi^n(x, dh)
\Big)
X^n_s(dx,da)\,ds\label{PBMx}\end{align}
is a square integrable martingale started at 0 with quadratic variation
\begin{align}
\langle M^{n,f}\rangle_t&=  \frac{1}{n}\int_0^t \int_{\X\times \R_+} \Big(
\big(nr(x,a)+b(x,a)\big)\int_{\R^d} f^2(x+h) K^n(x,dh)\nonumber\\
&\hskip1cm + \big(nr(x,a)+ d(x,a)+X^n_sU(x,a)\big)f^2(x)\Big)
X^n_s(dx,da)\,ds\nonumber\\
&=  \frac{1}{n}\int_0^t \int_{\X\times \R_+} \Big(
\big(2nr(x,a)+b(x,a)+ d(x,a)+X^n_sU(x,a)\big) f^2(x)
\nonumber\\
&\hskip1cm + \big(nr(x,a)+ b(x,a)\big) p(x,a) \int_{\R^d} \big(f^2(x+h) -f(x)\big)\pi^n(x,dh)\Big)
X^n_s(dx,da)\,ds.\label{crochetPBMx}
\end{align}

\noindent We divide the proof into several steps.\\

\noindent \textbf{Step 1} Firstly, we prove the uniform tightness of $(\mathcal{L}(\bar{X}^n))_{n\in \N^*}$ in the space of probability measures on $\D([0,T],(\mathcal{M}_F(\X),v))$.\\
Let us consider a continuous bounded function $f\in \mathcal{D}(A)$, and  show the  uniform tightness of
the sequence $(\langle \bar{X}^n_.,f\rangle)_{n\in \N^*}$  in $\D([0,T],\R)$. We remark that  for every fixed $t\in [0,T]$ and  $n>0$,
\begin{align}
\mathbb{P}(|\langle \bar{X}^n_t,f\rangle|>k)\leq \frac{\|f\|_\infty \, \E(\sup_{t\in [0,T]}\langle \bar{X}^n_t,1\rangle)}{k},\label{etape18}
\end{align}which tends to 0 as $k$ tends to infinity (cf.  \eqref{estimee_momentp}). This proves the tightness of the family of time-marginals $(\langle \bar{X}^n_t,f\rangle)_{n\in \N^*}$.
\noindent Denoting by
$A^{n,f}$  the finite variation process in the r.h.s. of (\ref{PBMx}) and thanks to Assumption \ref{tauxmutation}, we get for all stopping times $S_n<T_n<(S_n+\delta)\wedge T$,
that
\begin{multline}
\E(|A_{T_n}^{n,f}-A_{S_n}^{n,f}|)\leq  \delta\, \Big[\Big( (\|Af(s,.)\|_{\infty}+1)
\frac{\bar{r}}{2}+\|f\|_\infty(\bar{b}+\bar{d})\Big)\sup_{n\in \N^*}\E(\sup_{t\in [0,T]}\langle X^n_t,1\rangle)\\
+\|f\|_\infty \bar{U}\sup_{n\in \N^*}\E(\sup_{t\in [0,T]}\langle X^n_t,1\rangle^2)\Big].\label{etape20}
\end{multline}
The quadratic variation process \eqref{crochetPBMx} satisfies a similar inequality:
\begin{align}
 & \mathbb{E}(|\langle M^{n,f}\rangle_{T_n}-\langle M^{n,f}\rangle_{S_n}|)\nonumber\\
 \leq  &  \|f\|_\infty^2 \E\Big(\int_{S_n}^{T_n} \Big[2\bar{r}  \langle \bar{X}^n_s,1\rangle
 + \frac{\bar{b} \langle X^n_s,1\rangle}{n}+\frac{\bar{d}\langle \bar{X}^n_s,1\rangle + \bar{U}\langle \bar{X}^n_s,1\rangle^2}{n}\Big]ds\Big)\nonumber\\
 \leq &  \|f\|_\infty^2 \delta \;\Big[\Big(2\bar{r} +\frac{\bar{b}+\bar{d}}{n}\Big) \sup_{n\in \N^*}
 \E(\sup_{t\in [0,T]}\langle \bar{X}^n_t,1\rangle)+\frac{\bar{U}}{n}\sup_{n\in \N^*}\E(\sup_{t\in [0,T]}\langle \bar{X}^n_t,1\rangle^2)\Big].\label{etape21}
\end{align}
Then, for $\varepsilon>0$, $\eta>0$, a sufficiently large $n$ and small $\delta$, we have using \eqref{estimee_momentp}  and Assumption \ref{tauxmutation}, that
\begin{align}
\mathbb{P}(|A_{T_n}^{n,f}-A_{S_n}^{n,f}|>\eta)\leq \varepsilon\quad \mbox{ and }\quad \mathbb{P}(|\langle M^{n,f}\rangle_{T_n}-\langle M^{n,f}\rangle_{S_n}|>\eta)\leq \varepsilon. \label{aldous}
\end{align}
From \eqref{etape18}, \eqref{etape20}, \eqref{etape21} and the Aldous-Rebolledo criterion (see \eg \cite{joffemetivier} or \cite[Th. 1.17]{etheridgebook}), we obtain the uniform tightness of the sequence $(\langle \bar{X}^n_.,f\rangle)_{n\in \N^*}$  in $\D([0,T],\R)$. Thanks to  Roelly's criterion  \cite{roelly}, we conclude that $(\bar{X}^n)_{n\in \N^*}$ is uniformly tight in $\D([0,T],(\mathcal{M}_F(\X),v))$.\\

Let us denote by $\bar{X}$ a limiting process of $(\bar{X}^n)_{n\in \N^*}$. It is
almost surely (a.s.) continuous in $(\mathcal{M}_F(\X),v)$ since
\begin{equation}
\sup_{t\in \R_+}\sup_{f,\,\|f\|_\infty\leq 1} |\langle \bar{X}^n_t,f\rangle -\langle \bar{X}^n_{t_-},f\rangle |\leq \frac{1}{n}.\label{la_limite_est_continue}
\end{equation}

\noindent In the case where $\X$ is a compact subset of $\mathbb{R}^d$, the vague and weak  topologies coincide, which fails in the non-compact case. Nevertheless  the tightness in  $\D([0,T],(\mathcal{M}_F(\X),w))$ is needed to identify the limiting values of $(\bar{X}^n)_{n\in \N^*}$.

\medskip
\noindent \textbf{Step 2} Let us now concentrate on the case where $\X$ is unbounded and let us show the tightness of $(\bar{X}^n)_{n\in \N^*}$ in $\D([0,T],(\mathcal{M}_F(\X),w))$.
The same computation as in Step 1 for $f(x)=1$ implies that the sequence $(\langle \bar{X}^n,1\rangle)_{n\in \N^*}$ is uniformly tight in $\D([0,T],\R_+)$. As a consequence, it is possible to extract from $(\bar{X}^n)_{n\in \N^*}$ a subsequence $(\bar{X}^{u_n})_{n\in \N^*}$ such that:
\begin{itemize}
\item $(\bar{X}^{u_n})_{n\in \N^*}$ converges in distribution to $\bar{X}$ in $\D([0,T],(\mathcal{M}_F(\X),v))$,
\item $(\langle \bar{X}^{u_n},1\rangle)_{n\in \N^*}$ converges in distribution in $\D([0,T],\R_+)$.

\end{itemize}Let us now show that the limit of $(\langle \bar{X}^{u_n},1\rangle)_{n\in \N^*}$ is $\langle \bar{X},1\rangle$, empedding  a loss of mass in the limit. Indeed, as a consequence, a criterion in  M\'{e}l\'{e}ard and Roelly \cite{meleardroelly} will prove that $(\bar{X}^{u_n})_{n\in \N^*}$ converges in distribution to $\bar{X}$ in $\D([0,T],(\mathcal{M}_F(\X),w))$.

\medskip \noindent By simplicity, we will  again denote $u_{n}$  by $n$.

\medskip \noindent As in Jourdain-M\'el\'eard \cite{JM09}, we introduce a sequence of smooth functions $\psi_k$  defined on $\mathbb{R}_{+}$ and approximating ${\bf 1}_{\{u \geq k\}}$. For $k\in{\mathbb N}$,  let $\psi_k(u)=\psi(0\vee(u-(k-1))\wedge 1)$ where $\psi(y)=6y^5-15 y^4+10y^3$ is a nondecreasing  function such that $\psi(0)=\psi'(0)=\psi''(0)=1-\psi(1)=\psi'(1)=\psi''(1)=0$.
The function $u \mapsto \psi_k(u)$ is  nondecreasing on $\mathbb{R}_+$, equals  $0$ on $[0,k-1]$ and  $1$
 on the complement of $[0,k)$. In particular $\psi_0\equiv 1$. Moreover the sequence $(\psi_k)_{k\in \N^*}$ is nonincreasing, and satisfies for $u\geq 0$ and $p\geq 1$ that
 \begin{eqnarray}
 \label{fk}
 {\bf 1}_{\{u\geq k\}}&\leq &\psi_k(u)\leq {\bf 1}_{\{u\geq k-1\}};\\
 \label{ineq}\psi_k^{(p)}(u)&\leq& \sup_{u \in  [k-1,k]} \, |\psi_k^{(p)}(u)|\, {\bf 1}_{\{u\geq k-1\}}\leq \sup_{u\in  [k-1,k]} \, |\psi_k^{(p)}(u)|\, \psi_{k-1}(u).\nonumber
 \end{eqnarray}
The proof of the following lemma is postponed at the end of Proposition \ref{proptightness}'s proof. We define $f_{k}(x) = \psi_{k}(\Vert x\Vert)$, for all $x\in \X$.

\begin{lemme}\label{lemmeproptightness}
\label{controle-masse}Under the assumptions of Proposition \ref{proptightness},
$$\lim_{k\to+\infty}\limsup_{n\to+\infty}\E\left(\sup_{t\leq T}\langle \bar{X}^n_t,f_k\rangle\right)=0.$$
\end{lemme}

\bi \noindent  From Lemma \ref{lemmeproptightness}, we can deduce that \begin{equation}
   \lim_{k\to+\infty}\E\Big(\sup_{t\leq T}\langle \bar{X}_t,f_k\rangle\Big)=0.\label{tightlim}
\end{equation}
 Indeed, for $k\in{\mathbb N}$, the continuous and compactly supported functions $(f_{k,\ell}\stackrel{\rm def}{=}f_k(1-f_\ell))_{\ell\in{\mathbb N}}$ increase to $f_k$ as $\ell\to+\infty$. By continuity of $\nu\mapsto \sup_{t\leq T} \langle \nu_t,f_{k,\ell}\rangle$ on $\mathbb{D}([0,T], (M_F,v))$ and uniform integrability deduced from the uniform square moment estimates \eqref{estimee_momentp}, one has
$$\E\Big(\sup_{t\leq T}\langle \bar{X}_t,f_{k,\ell}\rangle\Big)=\lim_{n\to+\infty}\E\Big(\sup_{t\leq T}\langle \bar{X}^n_t,f_{k,\ell}\rangle\Big)\leq\liminf_{n\to + \infty}\E\Big(\sup_{t\leq T}\langle \bar{X}^n_t,f_{k}\rangle\Big).$$
Taking the limit $\ell\to + \infty$ in the left-hand-side by the monotone convergence theorem, one concludes   that for $k=0$
\begin{equation}
   \E\Big(\sup_{t\leq T}\langle \bar{X}_t,1\rangle\Big)=\E\Big(\sup_{t\leq T}\langle\bar{X}_t,f_0\rangle\Big)<+\infty,\label{contmass}
\end{equation} and   from  Lemma \ref{controle-masse}, that \eqref{tightlim} holds for any $k$.\\
As a consequence one may extract a subsequence of $(\sup_{t\leq T}\langle \bar{X}_t,f_k\rangle)_k$ that converges to $0$ a.s.,
and since the process $(\bar{X}_t)_{t\leq T}$ is  continuous   from $[0,T]$ into $(\mathcal{M}_F(\X),v)$, one deduces that it is also continuous   from $[0,T]$ into $(\mathcal{M}_F(\X),w)$.\\
We can now prove the convergence of $\langle X^{u_n},1\rangle$ to $\langle \bar{X},1\rangle$.
For $F$ a Lipschitz continuous and bounded function on  $\mathbb{D}([0,T],  \mathbb{R})$, we have
\begin{align*}
\limsup_{n\to + \infty}|\E\big(F(\langle  \bar{X}^n,1\rangle)&- F(\langle  \bar{X},1\rangle)\big)|
\leq \limsup_{k\to + \infty}\limsup_{n\to + \infty}|	\E\big(F(\langle  \bar{X}^n,1\rangle)- F(\langle  \bar{X}^n,1-f_k\rangle)\big)| \\
&+
\limsup_{k\to + \infty}\limsup_{n\to + \infty}|\E\big(F(\langle  \bar{X}^n,1-f_k\rangle)- F(\langle  \bar{X},1-f_k\rangle)\big)| \\
&+\limsup_{k\to + \infty}|\E\big(F(\langle  \bar{X},1-f_k\rangle)- F(\langle  \bar{X},1\rangle)\big)|.
\end{align*}
Since $|F(\langle \nu,1-f_k\rangle)- F(\nu ,1\rangle)|\leq C\sup_{t\leq T}\langle \nu_t,f_k\rangle$ by Lipschitz property,  the first and the third terms in the r.h.s. are equal to $0$ respectively according to Lemma \ref{controle-masse} and to \eqref{tightlim}. The second term is $0$ by continuity of $\nu\mapsto \langle\nu,1-f_k\rangle$ in $\mathbb{D}([0,T],(\M_F(\X),v))$. \\
This ends the proof of Proposition \ref{proptightness}.
  \end{proof}

\medskip
\noindent
\begin{proof}[Proof of Lemma \ref{controle-masse}]
Firstly, let us show that for each $t\in [0,T]$,
\be
\label{tfixe} \lim_{k\to + \infty}\limsup_{n\to + \infty}\E\left(\langle \bar{X}^n_t,f_k\rangle\right)=0.\ee
The boundedness of $r$  and Assumption \ref{tauxmutation}-2 ensure the existence of a sequence $(\varepsilon_n)_{n\in \N^*}$ converging to $0$ such that
\begin{align}
    \E( \langle\bar{X}^n_t,f_k\rangle)&\leq \E(\langle  \bar{X}^n_0,f_k\rangle)+\bar{b}\int_0^t \E(\langle  \bar{X}^n_s,f_k\rangle) ds+\varepsilon_n\int_0^t \E(\langle  \bar{X}^n_s,1\rangle) ds\notag\\
    &+\E\Big(\int_0^t\int_{\R^d\times \R_+}   r(x,a)p(x,a) \,A f_k(x) X^n_s(dx,da)ds\Big),
\end{align}
and we have by \eqref{estA} and \eqref{fk}
 $$\Big|\int_{\mathbb{R}^d \times \R_+} r(x,a)p(x,a)\, Af_k(x) X^n_s(dx,da)\Big|\leq \bar{r} \sum_{|\ell|\leq \ell_0} \|D^{\ell}f_{k}\|_{\infty}\ \langle  \bar{X}^n_s,f_{k-1}\rangle.$$
\noindent Since moreover, the sequence $(f_k)_{k\in \N^*}$ is non-increasing,  $\langle  \bar{X}^n_s,f_k\rangle \leq \langle  \bar{X}^n_s,f_{k-1}\rangle$ and
 there is a constant $C>0$ independent of $k\geq 2$ such that
\begin{align}
   \E(\langle  \bar{X}^n_t,f_k\rangle)&\leq \E(\langle  \bar{X}^n_0,f_k
   \rangle) +C\int_0^t \E(\langle \bar{X}^n_s,f_{k-1}\rangle) ds+\varepsilon_n\int_0^t\E(\langle  \bar{X}^n_s,1\rangle ) ds.\end{align}
Let $\mu^{n,k}_s=\E\left(\langle  \bar{X}^n_s,f_k\rangle\right)\leq \mu^n_{s}=\E\left(\langle  \bar{X}^n_s,1\rangle\right)$ which is bounded uniformly in $n\in \N^*$ and $s\in[0,T]$ according to \eqref{estimee_momentp}.
There exist two positive constants $C_1$ and $C_2$ such that
\begin{align*}
 \mu^{n,k}_t\leq \mu^{n,k}_0+C_{1}\int_0^t\mu^{n,k-1}_s ds+C_2\varepsilon_n.\end{align*}
 Iteration of this inequality yields \begin{align*}
   \mu^{n,k}_t &\leq \sum_{\ell=0}^{k-1
   }\mu^{n,(k-\ell)}_0\frac{(C_{1}t)^\ell}{\ell !}+\frac{(C_{1}\int_0^t\mu^n_sds)^{k}}{k!}+\varepsilon_n\, C_{2}\, \sum_{\ell=0}^{k-1}\frac{(C_{1}t)^\ell}{\ell !}\\
&\leq \mu^{n,\lfloor k/2\rfloor}_0e^{C_{1}t}+\mu^n_0\sum_{\ell=\lfloor k/2\rfloor +1}^{+\infty}\frac{(C_{1}t)^\ell}{\ell !}+\frac{(C_{1}'t)^{k}}{(k)!}+\varepsilon_n\, C_{2}\,e^{C_{1}t}.
\end{align*}
where we used the monotonicity of $\mu^{n,k}_0$ w.r.t. $k$ for the second inequality.
Given the moment condition (\ref{moment}), the assumption of tightness in $(\mathcal{M}_F(\X),w)$ of the initial conditions $(\bar{X}^n_0)_{n\in \N^*}$ is equivalent to
\begin{equation}
\lim_{k\rightarrow +\infty} \limsup_{n\rightarrow +\infty} \mu^{n,k}_0=0.
\end{equation}
Hence
$$\lim_{k\rightarrow +\infty}\limsup_{n\to + \infty}\mu^{n,k}_t\leq
\sup_{n\in \N^*} \mu^n_0 \lim_{k\rightarrow +\infty}
\sum_{\ell=\lfloor k/2\rfloor +1}^{+\infty}\frac{(C_{1}t)^\ell}{\ell !}+\lim_{k\rightarrow +\infty}\frac{(C_{1}'t)^{k}}{(k)!}.
$$
We deduce immediately that
\begin{equation}\lim_{k\to + \infty}\limsup_{n\to + \infty}\E\left(\langle \bar{X}^n_t,f_k\rangle\right)=\lim_{k\to + \infty}\limsup_{n\to + \infty}\mu^{n,k}_t =  0.\end{equation}

   Let us now consider  the martingale $M^{n,k}_t$  defined by \eqref{PBMx} with $f_k$ instead of $f$, and with quadratic variation given in \eqref{crochetPBMx}.  Similar arguments as above allow us to prove that
   $$ \E(\langle M^{n,k}\rangle_t)\leq C_{1}  \int_0^t \E(\langle X^n_s,f_{k-1}\rangle) ds +\varepsilon_{n} C_{2} \int_0^t \E(\langle X^n_s,1\rangle) ds.$$
   Thus, using
   that $f_k\leq 1$, Doob's inequality,   \eqref{tfixe},  \eqref{estimee_momentp} and the dominated convergence theorem, we get
   $$ \lim_{k\to + \infty}\limsup_{n\to + \infty}\E\Big(\sup_{{t\leq T}} \vert M^{n,k}_t\vert\Big)=0.$$
  Let us now come back to the process $\langle  \bar{X}^n,f_k\rangle$.
  As before, we can get
     \begin{align}
   \langle  \bar{X}^n_t,f_k\rangle&\leq \langle  \bar{X}^n_0,f_k\rangle + M_t^{n,k}+\bar{b}\int_0^t\langle  \bar{X}^n_s,f_k\rangle ds+\varepsilon_n\int_0^t\langle  \bar{X}^n_s,1\rangle ds\notag\\
   &+\int_0^t\int_{\R^d\times \R_+}   r(x,a)p(x,a) \,A f_k(x) X^n_s(dx,da)ds \nonumber\\
&\leq \langle  \bar{X}^n_0,f_k\rangle + M_t^{n,k}+C_{1}\int_0^t\langle  \bar{X}^n_s,f_{k-1}\rangle ds+\varepsilon_n\, C_{2}\, \int_0^t\langle  \bar{X}^n_s,1\rangle ds,\end{align}
   for constants $C_{1}$ and $C_{2}$.
Let $\alpha^{n,k}_t=\E\left(\sup_{s\leq t}\langle  \bar{X}^n_s,f_k\rangle \right)$ and $\alpha^n_t=\E\left(\sup_{s\leq t}\langle  \bar{X}^n_s,1\rangle\right)$ which is bounded uniformly in $n\in \N^*$ and $t\in[0,T]$ according to \eqref{estimee_momentp}. One deduces that
\begin{align*}
 \alpha^{n,k}_t\leq \alpha^{n,k}_0+C_{1}\int_0^t\mu^{n,k-1}_s ds+C_2\varepsilon_n+\E\Big(\sup_{t\leq T} \vert M^{n,k}_t\vert \Big).\end{align*}
An iteration as before allows us to prove that
$$\lim_{k\to + \infty}\limsup_{n\to + \infty} \E\Big(\sup_{t\leq T} \langle \bar{X}^n_{t}, f_{k}\rangle\Big) = \lim_{k\to + \infty}\limsup_{n\to + \infty} \alpha^{n,k}_t =0,$$
which concludes the proof of Lemma \ref{lemmeproptightness} and thus the one of Proposition \ref{proptightness}.
\end{proof}

\subsection{Identification of the limiting values}\label{sectionidentification}

To obtain the convergence stated in Theorem \ref{propconvergence}, we  show that the limiting value $\bar{X}$ of the uniformly tight sequence $(\bar{X}^n)_{n\in \N^*}$ is unique. We establish a martingale problem satisfied by $\bar{X}$ in which there are integration terms with respect to the equilibrium (\ref{mchapeau}) involved in the averaging phenomenon for the ages. The uniqueness of the solution to the martingale problem is then proved.

\subsubsection{Averaging phenomenon}\label{sectionaveragingphenomenon}

\noindent We begin with establishing the form of the limiting values of the time-marginal distributions $(X^n_t(dx,da))_{n\in \N^*}$ for $t\in [0,T]$. Since the sequence $(\bar{X}^{n})_{n\in \N^*}$ is
 uniformly tight, there exists a subsequence of $(X^n_t(dx,da))_{n\in \N^*}$, with trait-marginals converging in law to a limiting value  $\bar{X}$, that by simplicity,  we denote again  by $(X^n_t(dx,da))_{n\in \N^*}$.\\

\noindent  We have already explained why  the uniform tightness of the sequence   $(t \mapsto X^n_t(dx,da))_{n\in \N^*}$ in $\mathbb{D}([0,T],\mathcal{M}_F(\X\times \R_+))$ cannot hold. However, following  Kurtz \cite{kurtzaveraging}, we will prove the uniform tightness of the sequence of random measures
\begin{equation}
\Gamma^n(dt,dx,da)=X^n_t(dx,da)dt\label{defGamman}
 \end{equation}on $\mathcal{M}_F([0,T]\times \X\times \R_+)$. Proceeding in this way allows us to escape the difficulties created by the degeneracies due to the rapid time scale for age, when one tries to follow individual paths.


\begin{prop}\label{prophomogeneisation2}
Under the assumptions of Theorem \ref{propconvergence}, the sequence $(\Gamma^n)_{n\in \N^*}$ converges in law to $\ \bar{X}_t(dx) \widehat{m}(x,da)dt\,$ in $\mathcal{M}_F([0,T]\times \X\times \R_+)$.\\
As a consequence, for $dt$-almost every (a.e.) $t\in [0,T]$, the sequence $(X^n_t(dx,da))_{n\in \N^*}$ converges weakly to $ \widehat{m}(x,da)\bar{X}_t(dx)$, with $ \widehat{m}$  defined in (\ref{mchapeau}).
\end{prop}

\noindent The proof of Proposition \ref{prophomogeneisation2} is inspired by   Kurtz \cite{kurtzaveraging}. To establish the the result, we need to consider the random measure defined in \eqref{defGamman}.\\
\noindent   Firstly, we prove the uniform tightness of  the sequence  $(X^n_t(dx,da))_{n\in \N^*}$, for fixed $t\in [0,T]$ (Lemma \ref{lemme1}), as well as the one of the sequence of measures $(\Gamma^n)_{n\in \N^*}$ (Lemma \ref{lemme2}), where the pathwise and individual  points of view have been forgotten. The techniques to disentangle the traits and individuals' time scales appear strikingly  in the proof of Lemma \ref{lemme1}, where different treatments are used for the trait marginal and for the ages, with the introduction of the individuals' lifelengths. Then, in the proof of Proposition  \ref{prophomogeneisation2},  a factor $n$ appears in \eqref{etape28}, when changing from the macroscopic  scale to the microscopic scale. The next part of the proof consists in identifying the limiting martingale problem.

\begin{lemme}\label{lemme1}
For $dt$-a.e. $t\in [0,T]$, the sequence $(X_t^n)_{n\in \N^*}$ is uniformly tight on $\mathcal{M}_F(\X\times \R_+)$.
\end{lemme}

\begin{proof}
Let $\varepsilon>0$. Since the family $(\bar{X}^n_t)_{n\in \N^*}$ is tight, there exists a compact set $K\subset \R^d$ such that
\begin{equation}
\sup_{n\in \N^*}\P\big(\bar{X}^n_t(K^c)>\varepsilon\big)<\varepsilon.\label{eq1}
\end{equation}
Moreover, thanks to Point 2 of Assumption \ref{hyptaux} and using a coupling argument,
the life-lengths of the individuals in the population $X^n_t$ born after time 0 are dominated, uniformly in $x\in \X$, by independent random variables $D^n_i(t)$ with survival function $S^n$ defined in \eqref{survie}. Because the aging velocity is $n$, the ages of these individuals satisfy $A_i(t)\leq n \, D^n_i(t)$. For an individual $i$ alive at time 0, conditionally on the state at time 0, the remaining time it has to live can be dominated by an independent random variable $\Delta^n_i$ such that
$$\P\big(\Delta^n_i>\ell\big) =\exp\Big(-\int_0^\ell n \underline{r}(A_i(0)+nu) du\Big).$$ Thus, for $A>0$ and $n_0, N\in \N^*$,
\begin{align}
\sup_{n\geq n_0}\P\big(X^n_t((K\times [0,A])^c)>3\varepsilon\big)\leq & 
\sup_{n\geq n_0}\P\big(\bar{X}^n_t(K^c)>\varepsilon\big)+\sup_{n\geq n_0}\P\Big({1\over n}\sum_{i=1}^{N^n_t}\ind_{\{A_i(t)>A\}}>2\varepsilon\Big)\nonumber\\
\leq & \varepsilon+\sup_{n\geq n_0}\P\Big(\sum_{i=1}^{nN}\ind_{\{n\,D^n_i(t)>A\}}>n\varepsilon\Big)+\sup_{n\geq n_0}\P\big(N^n_t>nN\big)\nonumber\\
 & + \sup_{n\geq n_0} \P\Big(\sum_{i=1}^{N^n_0} \ind_{\{A_i(0)+n\Delta^n_i>A\}}>n\varepsilon\Big).\label{etape8}
\end{align}By (\ref{moment}), it is possible to
find $N$ such that:
\begin{equation}
\sup_{n\geq n_0}\P\big(N^n_t>nN\big)=\sup_{n\geq n_0}\P\big(\langle X^n_t,1\rangle>N\big)\leq
\frac{\sup_{n\geq n_0}\E\big(\sup_{t\in [0,T]}\langle X^n_t,1\rangle)}{N}\leq \varepsilon.\label{tightness1}
\end{equation}Let us consider the second term of \eqref{etape8}. Notice that
\begin{align*}
\E\big(\ind_{\{n\,D^n_i(t)>A\}}\big)=S^n(A/n)=\exp\Big(-\int_0^A \underline{r}(a)da\Big),
\end{align*}which converges to 0 when $A$ tends to infinity by \eqref{conditionminorationr}.
For $N$ as in \eqref{tightness1}, there hence exists $A_1$ sufficiently large so that for $A> A_1$, $S^n(A/n)<\varepsilon/2N$. Then
\begin{multline}
\P\Big(\sum_{i=1}^{nN}\ind_{\{D^n_i(t)>A/n\}}>n\varepsilon\Big)=
\P\Big(\sum_{i=1}^{nN}(\ind_{\{D^n_i(t)>A/n\}}-S^n(A/n))>n(\varepsilon-NS^n(A/n))\Big)\\
\leq
\P\Big(\sum_{i=1}^{nN}(\ind_{\{D^n_i(t)>A/n\}}-S^n(A/n))>n\varepsilon/2\Big)
\leq  \exp\left(-\frac{n\varepsilon^2}{8(N e^{-\int_0^A \underline{r}(a)da}(1-e^{-\int_0^A \underline{r}(a)da})+\varepsilon/3)}\right)\label{etape9}
\end{multline}by Bernstein's inequality (\eg \cite{shorackwellner} p.855). For a sufficiently large $n_1$ and for $n\geq n_1$, the r.h.s. of (\ref{etape9}) is smaller than $\varepsilon$.\\
Let us now upper bound the last term of \eqref{etape8}. Let $A\mapsto \phi(A)$ be such that $\phi(A)$ and $\int_{\phi(A)}^A \underline{r}(u)du$ tend to infinity when $A$ tends to infinity. We firstly notice that:
\begin{align}
\frac{1}{n}\sum_{i=1}^{N^n_0} \E\big(\ind_{\{A_i(0)+n\Delta^n_i>A\}}\, |\, X^n_0\big)= & \frac{1}{n}\sum_{i=1}^{N^n_0} 1\wedge e^{-\int_{A_i(0)}^A \underline{r}(u)du}=\int_{\X\times \R_+} 1\wedge e^{-\int_{a}^A \underline{r}(u)du} X^n_0(dx,da)\nonumber\\
\leq &  X^n_0(\X\times [\phi(A),+\infty))+ \langle X^n_0,1\rangle e^{-\int_{\phi(A)}^A \underline{r}(u)du}.\label{etape10}
\end{align} Recall that $X_0$ is the limit, in probability and for the weak convergence, of $(X^n_0)_{n\in \N^*}$. There hence exists $n_2(A)$ sufficiently large such that for $n\geq n_2(A)$, the right hand side of \eqref{etape10} is smaller than $$X_0(\X\times [\phi(A),+\infty))+ \langle X_0,1\rangle e^{-\int_{\phi(A)}^A \underline{r}(u)du}+\frac{\varepsilon}{4}$$with probability $1-\varepsilon$. By choice of $\phi(A)$, there exists $A_2$ sufficiently large such that such that for $A> A_2$, and $n\geq n_2(A_2)$ this upper bound is smaller than $\varepsilon/2$. For such $A$ and $n$, we have
\begin{align}
\P\Big(\sum_{i=1}^{N^n_0} \ind_{\{A_i(0)+n\Delta^n_i>A\}}>n\varepsilon\Big)
  \leq & \E\Big(\P\Big(\frac{1}{n}\sum_{i=1}^{N^n_0} \big(\ind_{\{A_i(0)+n\Delta^n_i>A\}}-1\wedge e^{-\int_{A_i(0)}^A \underline{r}(u)du}\big)>\frac{\varepsilon}{2}\, |\, X^n_0\Big)\Big)+\varepsilon\nonumber\\
  \leq & \E\Big(\exp\Big(-\frac{n \varepsilon^2}{2\langle X^n_0,1\rangle+8\varepsilon/3}\Big)\Big)+\varepsilon.\label{etape11}
\end{align}by applying Bernstein's inequality again. There exists $n_3$ sufficiently large so that for every $n\geq n_3$, the right hand side of \eqref{etape11} is smaller than $2\varepsilon$.\\
The tightness of $(X^n_t)_{n\in \N^*}$ is thus a consequence of (\ref{etape8}), (\ref{tightness1}), (\ref{etape9}) and \eqref{etape11}, with the choices of $A\geq \max(A_1,A_2)$ and $n_0\geq \max(n_1,n_2,n_3)$.
\end{proof}


\begin{lemme}\label{lemme2}
The family $(\Gamma^n)_{n\in \mathbb{N}}$ is tight in
$\mathcal{M}_F([0,T]\times \X\times \R_+)$.
\end{lemme}

\begin{proof}
Following Kurtz \cite[Lemma 1.3]{kurtzaveraging}, a sufficient condition for the tightness of the family $(\Gamma^n)_{n\in \mathbb{N}}$ is that for all $\varepsilon>0$, there exists
a compact set $\Xi$ of $\X\times \R_+$ such that
\begin{equation}
\sup_{n\in \N^*}\E\Big(\Gamma^n\big([0,T]\times \Xi^c\big)\Big)\leq C(T) \varepsilon.\label{etape14}
\end{equation}
Let us establish (\ref{etape14}). From the proof of Lemma \ref{lemme1}, it appears that the upperbounds (\ref{eq1}),
(\ref{etape8}) and (\ref{etape9}) are uniform in $t\in [0,T]$ so that:
\begin{equation}
\sup_{t\in [0,T]}\sup_{n\in \N^*}\mathbb{P}\Big(X^n_t\big((K\times [0,A])^c\big)>3\varepsilon\Big)<5\varepsilon.
\end{equation}
We are now ready to upperbound
\begin{align*}
\E\Big(\Gamma^n\big([0,T]\times (K\times [0,A])^c\big)\Big)= & \E\Big(\int_0^T \langle X^n_t, \ind_{(K\times [0,A])^c} \rangle dt\Big)
=\int_0^T \E\Big(X^n_t\big((K\times [0,A])^c\big)\Big)dt.
\end{align*}
Indeed:
\begin{multline}
\E\Big(X^n_t\big((K\times [0,A])^c\big)\Big)\leq  3\varepsilon\ \P\Big(X^n_t\big((K\times [0,A])^c\big)\leq 3\varepsilon\Big)+\E\Big(\langle X^n_t,1\rangle
\ind_{X^n_t\big((K\times [0,A])^c\big)> 3\varepsilon}\Big)\\
\leq  3\varepsilon +\sqrt{\E\Big(\langle X^n_t,1\rangle^2\Big)}\sqrt{\P\Big(X^n_t\big((K\times [0,A])^c\big)> 3\varepsilon\Big)}
\leq C(T)(\varepsilon+\sqrt{\varepsilon}),
\end{multline}by Cauchy-Schwarz inequality and (\ref{estimee_momentp}). This proves (\ref{etape14}) and finishes the proof.
\end{proof}

\bigskip \noindent
Before proving Proposition \ref{prophomogeneisation2}, we provide a lemma characterizing $\widehat{m}(x,a)$.
\begin{lemme}\label{lemmehomogeneisation}
Let $x\in \X$ be fixed. There exists a unique probability measure $\widehat{m}(x,da)$ on $\mathbb{R}_{+}$, solution of the following equation: For $\psi\in \Co_c^{1}(\R_+,\R)$ with compact support in $[0,+\infty)$,
\begin{align}
\int_{\R_+}\partial_a \psi(a)\widehat{m}(x,da)=\int_{\R_+}\psi(a)r(x,a)\widehat{m}(x,da)-\psi(0)\int_{\R_+}r(x,a)\widehat{m}(x,da).\label{equationstationnaire}
\end{align} The probability measure $\widehat{m}(x,da)$ is absolutely continuous with respect
to the Lebesgue measure and its density is given in (\ref{mchapeau}).
\end{lemme}

\begin{proof}
Let us consider the test function $\psi(a)=\int_a^{+\infty} f(\alpha)d\alpha$, where $f\in \Co_c(\R_+,\R_+)$ is positive.
Then $\partial_a \psi(a)=-f(a)$ and $\psi(0)=\int_0^{+\infty} f(\alpha)d\alpha$. Equation (\ref{equationstationnaire}) gives by Fubini's theorem:
\begin{align}
\int_{\R_+}f(a)\widehat{m}(x,da)= & \int_{\R_+}\int_0^a f(\alpha)d\alpha \ r(x,a)\widehat{m}(x,da)\nonumber\\
= & \int_{\R_+}f(\alpha) \int_\alpha^{+\infty}r(x,a)\widehat{m}(x,da)\, d\alpha.
\end{align}This entails that $\widehat{m}(x,da)$ is absolutely continuous with respect to the Lebesgue measure with density
$\widehat{m}(x,a)=\int_{a}^{+\infty}r(x,\alpha)\widehat{m}(x,\alpha)d\alpha$. The latter implies that $a\mapsto \widehat{m}(x,a)$
is a function of class $\Co^1$. Using further an integration by part in (\ref{equationstationnaire}), and since $\widehat{m}(x,a)$ tends to 0 when $a$ grows to infinity, we get for all $\psi\in \Co^1_c(\R_+,\R)$
\begin{align}
-\psi(0)\widehat{m}(x,0)-\int_{\R_+}\psi(a)\partial_a \widehat{m}(x,a)da=\int_{\R_+}\big(\psi(a)-\psi(0)\big)r(x,a)\widehat{m}(x,a)da.
\end{align}By identification,  we obtain that $\widehat{m}(x,a)$ is a solution of
\begin{align}
 & \partial_a \widehat{m}(x,a)=-r(x,a)\widehat{m}(x,a)\nonumber\\
 & \widehat{m}(x,0)=\int_{\R_+}r(x,a)\widehat{m}(x,a)da,\label{PDEstationnaire}
\end{align}which is solved by
\begin{align}
\widehat{m}(x,a)=\widehat{m}(x,0)\exp\big(-\int_0^a r(x,\alpha)d\alpha\big).\label{etape6}
\end{align}Since $\widehat{m}(x,a)da$ is a probability measure, necessarily \begin{align}
\widehat{m}(x,0)=\int_{\R_+}r(x,a)\widehat{m}(x,a)da=\frac{1}{\int_{\R_+}\exp\big(-\int_0^a r(x,\alpha)d\alpha\big)da}.\label{etape7}
\end{align}This provides existence and uniqueness of the solution of (\ref{PDEstationnaire}) and hence of (\ref{equationstationnaire}).
\end{proof}

\begin{rque}Notice that the system (\ref{PDEstationnaire}) defines the stable age equilibrium of the McKendrick-Von Foerster equation \cite{mckendrick,vonfoerster} (see also \cite{webb}) when the birth and death rates equal to $r(x,a)$ and the trait $x$ is fixed.\hfill $\Box$
\end{rque}

\begin{rque}\label{rque-dense}The space $\Co^1_c(\R_+,\R)$ is separable and there exists a denumberable dense family $(\psi_k)_{k\in \N}$ in this set. To obtain the result of Lemma \ref{lemmehomogeneisation}, it is sufficient to have \eqref{equationstationnaire} for these functions $(\psi_k)_{k\in \N}$.
\end{rque}

\bigskip \noindent
We are now able to prove Proposition \ref{prophomogeneisation2}.

\begin{proof}[Proof of Proposition \ref{prophomogeneisation2}]From (\ref{defGamman}), we can see that the marginal measure
of $\Gamma^n(ds,dx,da)$ on $[0,T]\times \X$ is $\bar{X}^n_s(dx)ds$. For any real bounded test function $\varphi\,:\,(s,x)\mapsto \varphi_s(x)$ on $[0,T]\times \X$,
\begin{equation}
\int_0^t \int_{\X\times \R_+} \varphi_s(x)\Gamma^n(ds,dx,da)=\int_0^t \langle \bar{X}^n_s,\varphi_s\rangle ds.\label{etape17}
\end{equation}
The sequence $(\bar{X}^n)_{n\in \N^*}$ is uniformly tight by Proposition \ref{proptightness}, as well as   $(\Gamma^n)_{n\in \N^*}$, by Lemma \ref{lemme2} (ii). Using Prohorov's theorem, we thus deduce that $\ (\Gamma^n(ds,dx,da),\bar{X}^n_s(dx)ds)_{n}$ is relatively compact and there exists a subsequence
that converges in distribution to a limiting value, say $(\Gamma(ds,dx,da),\bar{X}_s(dx)ds)$. Taking (\ref{etape17}) to the limit, we obtain that
$\bar{X}_s(dx)ds$ is necessarily the marginal measure of $\Gamma(ds,dx,da)$ on $[0,T]\times \X$ up to a null-measure set. We deduce from this
(\eg Lemma 1.4 of Kurtz \cite{kurtzaveraging}) that there exists a (random) probability-valued process $(\gamma_{s,x}(da),\, s\in [0,T],x\in \X)$ that is predictable in $(\omega,s)$ and such that for all bounded measurable
function $\varphi(s,x,a)$ on $[0,T]\times \X\times \R_+$,
\begin{equation}
\int_0^t \int_{\X\times \R_+} \varphi(s,x,a)\Gamma(ds,dx,da)=\int_0^t \int_\X \int_{\R_+} \varphi(s,x,a)\gamma_{s,x}(da)\bar{X}_s(dx)ds.
\end{equation}
We now want to characterize the limiting value $\Gamma(ds,dx,da)=\gamma_{s,x}(da)\bar{X}_s(dx)ds$.
Applying (\ref{PBM}) for a test function $\varphi(x,a)\in \Co^{0,1}_b(\X\times \R_+,\R)$ and dividing by $n$ gives that:
\begin{align}
\frac{M^{n,\varphi}_t}{n}&=  \frac{\langle X^n_t,\varphi\rangle-\langle X^n_0,\varphi\rangle}{n}-\int_0^t \int_{\X \times \R_+}\Big[
\partial_a \varphi(x,a) + r(x,a)\Big(\varphi(x,0)
  -  \varphi(x,a)\Big)\nonumber\\
&\hskip 1cm +  p(x,a)\,\Big(r(x,a)+\frac{b(x,a)}{n}\Big)\int_{\R^d}\big(\varphi(x+h,0)-\varphi(x,0)\big)\pi^n(x, a, dh)\nonumber
\\
&\hskip 1cm  + \frac{b(x,a)}{n}\varphi(x,0)  -  \Big(\frac{d(x,a)+X^n_sU(x,a)}{n}\Big)\varphi(x,a)\Big]\Gamma^n(ds,dx,da)\label{etape28}
\end{align} is a martingale.
Using \eqref{estimee_momentp}, we can easily  prove that $\left(M^{n,\varphi}_t/n\right)_{n\in\mathbb{N}^*, t\in [0,T]}$ is uniformly integrable and that
\begin{equation*}
\lim_{n\rightarrow +\infty} \E\Big(\Big|\frac{M^{n,\varphi}_t}{n}-\widetilde{M}^{\varphi}_t\Big|\Big)=0,
\end{equation*}
where
\begin{equation}
\widetilde{M}^{\varphi}_t = \int_0^t \int_{\X\times \R_+}\Big[\partial_a \varphi(x,a)
+r(x,a)\big(\varphi(x,0)-\varphi(x,a)\big)\Big]\gamma_{s,x}(da)\bar{X}_s(dx)ds.\label{etape16}
\end{equation}

\noindent Moreover, the uniform integrability of $(M^{n,\varphi}/n)$ also provides that the process $(\widetilde{M}^{\varphi}_t )_{t}$ is a martingale. As it is also
a continuous and finite variation process, it must hence be almost surely zero. Since this holds for every $t\in \R_+$, we have proved that a.s., $dt$-a.e.
\begin{align}
\int_{\X\times\R_+}\Big[\partial_a \varphi(x,a)
+r(x,a)\big(\varphi(x,0)-\varphi(x,a)\big)\Big]\gamma_{t,x}(da)\bar{X}_t(dx)=0.\label{etape15}
\end{align}Choosing $\varphi(x,a)=\phi_\ell(x)\psi_k(a)$ with $(\phi_\ell)_{\ell\in \N}$ and $(\psi_k)_{k\in \N}$ dense families in $\Co_c(\X,\R)$ and $\Co^1_c(\R_+,\R)$, respectively, we obtain from
(\ref{etape15}) that, for all $\ell, k\ \in \N$, $$\int_{\X}\phi_\ell(x) H_k(t,x)\bar{X}_t(dx)=0$$where $
H_k(t,x)=\int_{\R_+}\Big[\partial_a \psi_k(a)
+r(x,a)\big(\psi_k(0)-\psi_k(a)\big)\Big]\gamma_{t,x}(da)$. Almost surely, the function $H_k(t,x)$ is  bounded and is thus $dt$-a.e. $\bar{X}_t(dx)$-integrable.
We obtain that for all $k\in \N$, a.s., $dt$-a.e.,  $\bar{X}_t(dx)$-a.e.,
\begin{align}
\int_{\R_+}\Big[\partial_a \psi_k(a)
+r(x,a)\big(\psi_k(0)-\psi_k(a)\big)\Big]\gamma_{t,x}(da)=0.
\end{align}
By Lemma \ref{lemmehomogeneisation} and Remark \ref{rque-dense}, we deduce that a.s., $dt$-a.e., and  $\bar{X}_t(dx)$-a.e., $\gamma_{t,x}(da) =
  \widehat{m}(x,a)da$ and as a consequence, any limiting value of $(X^n_{t})_{n\in \N^*}$
is of the form $\bar{X}_t(dx)\otimes \widehat{m}(x,a)da$. 
\end{proof}

\subsubsection{Characterization of the limiting values}\label{sectioncharacterization}

In the previous sections, we have proved that the sequence $(\bar{X}^n)_{n\in \N^*}$ is tight and that for a given limiting value $\bar{X}$, the associated subsequence $(\Gamma^n(dt,dx,da)=X^n_t(dx,da)\ dt)_{n\in \N^*}$ converges in $(\mathcal{M}_F([0,T]\times \X\times \R_+),w)$ to $\bar{X}_t(dx)\widehat{m}(x,da)\ dt$.
Now, we are ready to prove that:
 \begin{lemme}The limiting values $\bar{X}$ of the sequence $(\bar{X}^n)_{n\in \N^*}$ are solution of the martingale problem (\ref{martingalepremiercas})-(\ref{crochetmartingalepremiercas}).
\end{lemme}
\begin{proof}
Let $0<s_1\leq \dots s_k<s<t$, and let us introduce for $Y\in \D(\R_+,\mathcal{M}_F(\X))$:
\begin{multline}
\Psi_{s,t}(Y)=  \phi_1(Y_{s_1})\dots \phi_k(Y_{s_k})\Big\{\langle Y_t,f\rangle -\langle Y_s,f\rangle-\int_s^t du\int_{\X}
Y_u(dx)\Big[\\
\widehat{(p\,r)}(x)Af(x)+\big(\widehat{b}(x)-\widehat{d}(x)- Y_s\widehat{U}(x)\big)f(x)\Big]\Big\},
\end{multline}where $\phi_1,\dots, \phi_k$ are bounded continuous functions on $\mathcal{M}_F(\X)$ and $f\in \mathcal{D}(A)$.
Our purpose is to prove that $\E\big( \Psi_{s,t}(\bar{X}) \big)=0$ for any limiting value $\bar{X}$ of $(\bar{X}^n)_{n\in \N^*}$.\\

\noindent Let $\bar{X}$ be a limiting value of $(\bar{X}^n)_{n\in \N^*}$ and let $(\bar{X}^{u_n})_{n\in \N^*}$ be a subsequence converging to $\bar{X}$. On the one hand, thanks to Proposition \ref{prophomogeneisation2}, \eqref{estimee_momentp} and \eqref{la_limite_est_continue}:
\begin{align}
\E\big(\Psi_{s,t}(\bar{X})\big)=\lim_{n\rightarrow +\infty}\E\Big(\phi_1(\bar{X}^{u_n}_{s_1})\dots \phi_k(\bar{X}^{u_n}_{s_k})\Big\{\langle \bar{X}^{u_n}_t,f\rangle -\langle \bar{X}^{u_n}_s,f\rangle-\int_s^t du\int_{\X\times \R_+}
X^{u_n}_u(dx,da)\Big[\nonumber\\
p(x,a)\,r(x,a)\,Af(x)+\big(b(x,a)-d(x,a)- \int_{\X\times \R_+}U((x,a),(y,\alpha))X^{u_n}_u(dy,d\alpha)\big)f(x)\Big]\Big\}\Big).\label{etape23}
\end{align}
On the other hand, the term under the expectation in the r.h.s. of (\ref{etape23}) equals:
\begin{align}
\phi_1(X^{u_n}_{s_1})\dots \phi_k(X^{u_n}_{s_k})\Big\{M^{u_n,f}_t-M^{u_n,f}_s + A_{u_n}+B_{u_n}\Big\},\label{etape25}
\end{align}where $M^{u_n,f}$ has been defined in (\ref{PBMx}) and where:
\begin{align*}
A_{u_n}= & \int_s^t du \int_{\X\times \R_+} X_u^{u_n}(dx,da)\ r(x,a)\Big[u_n \int_{\X}\Big(f(x+h)-f(x)\Big)K^{u_n}(x, a, dh)-p(x,a)Af(x)\Big]\nonumber\\
 = & \int_s^t du \int_{\X\times \R_+} X_u^{u_n}(dx,da)\ r(x,a)\,p(x,a)\ \Big[u_n \int_{\X}\Big(f(x+h)-f(x)\Big)\pi^{u_n}(x,dh)-Af(x)\Big]\\
 B_{u_n}= & \int_s^t du \int_{\X\times \R_+} X_u^{u_n}(dx,da)\ b(x,a)\Big[\int_{\R^d}f(x+h) K^{u_n}(x, a, dh)-f(x)\Big]
\end{align*}
Firstly, using (\ref{estimee_momentp}) and the fact that the process $M^{n,f}$ is a martingale we obtain that:
\begin{equation}
\E\big(\phi_1(X^{u_n}_{s_1})\dots \phi_k(X^{u_n}_{s_k}) \big[M^{u_n,f}_t-M^{u_n,f}_s\big]\big)=0.\label{etape26}
\end{equation}Secondly, from Assumption \ref{tauxmutation}, $|\int_{\R^d}f(x+h)K^n(x, a, dh)-f(x)|=o(1/n)$ and using (\ref{estimee_momentp}) again provides:\begin{equation}
\lim_{n\rightarrow +\infty}\E\big(\phi_1(X^{u_n}_{s_1})\dots \phi_k(X^{u_n}_{s_k}) \big[A_{u_n}+B_{u_n}\big]\big)=0.\label{etape27}
\end{equation}
From \eqref{etape23}, \eqref{etape25}, \eqref{etape26} and \eqref{etape27}, we deduce that $\E\big(\Psi_{s,t}(\bar{X})\big)=0$ and hence $(M^f_t)_{t\in \R_+}$ defined in \eqref{martingalepremiercas} is a martingale. From this, using Itô's formula with localization arguments and Proposition \ref{prop_usefulresults} (i), we obtain that the following process is a martingale:
\begin{multline}
\langle \bar{X}_t,f\rangle^2-\langle \bar{X}_0,f\rangle^2-\int_0^t  2\langle \bar{X}_s,f\rangle \int_{\mathcal{X}} \Big[\widehat{(pr)}(x)Af(x)+\big(\widehat{b}(x)\\
-\widehat{d}(x)-\bar{X}_sU(x)\big)f(x)\Big]\bar{X}_s(dx)\ ds -\langle M^f\rangle_t\label{etape24}
\end{multline}is a martingale. Moreover, using the results of Proposition \ref{prop_usefulresults}, we obtain that the following process is a martingale:
\begin{multline*}
\langle \bar{X}^n_t  ,f\rangle^2-  \langle \bar{X}^n_0,f\rangle^2\\
\begin{aligned}
- & \int_0^t \Big[ 2 \langle \bar{X}^n_s,f\rangle \int_{\mathcal{X}\times \R_+}\Big(nr(x,a)\int_{\R^d}\big(f(x+h)-f(x)\big)K^n(x,a,dh)\\
 & \qquad + b(x,a)\int_{\R^d}f(x+h)K^n(x,a,dh)-\big(d(x,a)+X^n_sU(x,a)\big)f(x)\Big)X^n_s(dx,da)\ \Big] ds\\
- & \int_0^t \int_{\mathcal{X}\times \R_+} \Big[r(x,a)\Big(\int_{\R^d}f^2(x+h)K^n(x,a,dh)+f^2(x)\Big)\\
 &\qquad  +\frac{b(x,a)}{n}\int_{\R^d}f^2(x+h)K^n(x,a,dh)+\frac{d(x,a)+X^n_sU(x,a)}{n}f^2(x)\Big] X^n_s(dx,da)\ ds.
\end{aligned}
\end{multline*}By using arguments similar as those in the beginning of the proof, we deduce that
\begin{multline}
\langle \bar{X}_t  ,f\rangle^2-  \langle \bar{X}_0,f\rangle^2- \int_0^t  2 \langle \bar{X}_s,f\rangle \int_{\mathcal{X}}\Big[\widehat{(pr)}Af(x)\\
 + \big(\widehat{b}(x)-d(x)-\bar{X}_sU(x)\big)f(x)\Big]\bar{X}_s(dx) ds-  \int_0^t \int_{\mathcal{X}} 2\widehat{r}(x)f^2(x) \bar{X}_s(dx)\ ds\label{etape19}
\end{multline}is a martingale. Comparing \eqref{etape24} and \eqref{etape19} yields the bracket of the martingale $M^f$. This ends the proof.
\end{proof}

\subsubsection{Uniqueness of the martingale problem}\label{sectionunicitepbm}

We have shown that the limiting values of the uniformly tight sequence $(\bar{X}^n)_{n\in \N^*}$ satisfy the martingale problem (\ref{martingalepremiercas})-(\ref{crochetmartingalepremiercas}). To conclude the proof of Theorem \ref{propconvergence}, it remains to prove the uniqueness of the solution of this martingale problem.

\begin{prop}\label{lemmeuniciteage}There is a unique solution to the martingale problem of Theorem \ref{propconvergence}.
\end{prop}

\begin{proof}
 We start with getting rid of the non-linearity by using Girsanov's formula (see Dawson \cite{dawson} Theorem  7.2.2).
There exists a probability measure    $\mathbb{Q}$ on the path space such that  for all $f\in \mathcal{D}(A)$:
\begin{align}
\widetilde{M}^f_t =  & M^f_t + \int_0^t \int_{\R^d}\Big(\widehat{b}(x)-\widehat{d}(x)-\bar{X}_sU(x)\Big)f(x) \bar{X}_s(dx)\,ds\nonumber\\
= & \langle \bar{X}_t,f\rangle - \langle \bar{X}_0,f\rangle - \int_0^t \int_{\R^d} \widehat{(p\,r)}(x)Af(x) \bar{X}_s(dx)\, ds\label{defMtilde}
\end{align}is a square integrable martingale with bracket (\ref{crochetmartingalepremiercas}).\\

\noindent The uniqueness of the solution the martingale problem (\ref{defMtilde})-(\ref{crochetmartingalepremiercas}) is proved by Roelly and Rouault \cite{roellyrouault}. It is based on the branching property of $\bar{X}$ under $\mathbb{Q}$ which allows us to  characterize the  Laplace functional $L_t f$ of $X_t$ by its cumulant $U_tf$:
\begin{equation}
 L_t(f)=\E\Big(e^{\langle \bar{X},f \rangle} \Big) = \E\Big(e^{\langle \bar{X}_0,U_tf \rangle}\Big).
\end{equation}
The latter is the unique positive solution of the following PDE:
\begin{equation}
 \frac{\partial u}{\partial t}(t,x)=Au(t,x)-\widehat{r}(x)u^2(t,x),\qquad u(0,x)=f(x),\label{edp_cumulant}
\end{equation}
(see \eg Pazy \cite[Th. 1.4 and 1.5 p.185 and 187]{pazy}).

\noindent From the uniqueness of the solution of (\ref{defMtilde})-(\ref{crochetmartingalepremiercas}), we deduce classically the uniqueness of the solution of (\ref{martingalepremiercas})-(\ref{crochetmartingalepremiercas}). (See  for example Evans-Perkins \cite{evansperkins}, Fitzsimmons \cite{fitzsimmons}).
\end{proof}

\bigskip
\noindent The proof of Theorem \ref{propconvergence} is now complete.


\section{Examples}\label{sectionexemples}

Let us develop and compare two examples, which only differ by the function $r(x,a)$.

\subsection{Logistic physical-age and size-structured population}\label{exemple_r(x)=1}

In M\'{e}l\'{e}ard and Tran \cite{meleardtran}, the following example for a population structured by age $a\in \R_+$ and size $x\in \X=[0,x_0]$ is considered:
\begin{align}
& b(x,a)=x(x_0-x)e^{-a}\ind_{[0,x_0]}(x)\mbox{ for }x_0>0,\nonumber\\
& d(x,a)=d_0,\quad U((x,a),(y,\alpha))=\eta(x_0-x),\label{ex2}
\end{align}with $x_0=4$, $d_0=1/4$ and $\eta=1.7$. Because reproduction needs energy, and since this energy depends on the size of the created offspring, very small or big individuals are disadvantaged. Individuals of intermediate size $x=2$ have the highest birth rate. The competition term in contrast favors bigger individuals. Hence there is a trade-off between competitiveness and reproduction. The decreasing exponential in age describes a senescence phenomenon: older individuals reproduce less than their young competitors.  In \cite{meleardtran}, partial differential equation limits, Trait substitution sequence  and Canonical equations are considered.
Here we consider the  superprocess approximation described in the above sections, with $r(x,a)=1$ and $\pi^n(x,dh)$ a centered Gaussian kernel with variance ${\sigma^2\over n}$ conditioned on  $[0,x_0]$, as in Example \ref{exemple_noyau}.

\me Computation gives  $\widehat{m}(x,a)=e^{-a}$ so that $\ X_t(dx,da)=\bar{X}_t(dx)\otimes e^{-a}\,da\ $ becomes in this particular case a product measure. As soon as the population survives, the age distribution "stabilizes" around an exponential distribution with parameter $1$, as  seen on the simulations of Figure \ref{simu1}. With the age distribution $\widehat{m}(x,a)=e^{-a}$, we get
\begin{align}
\widehat{b}(x)=x(x_0-x)\int_{\R_+}e^{-2a}da=\frac{x(x_0-x)}{2} \quad ; \quad \widehat{d}(x)=d_0\quad \widehat{U}(x,y)=\eta(x_0-x).\label{bdUchapeau-exemple1}
\end{align}The martingale problem (\ref{martingalepremiercas}) becomes here:
\begin{align}
 & M^{f}_t=    \langle \bar{X}_t,f\rangle -\langle \bar{X}_0,f\rangle
-\int_0^t \int_{\X} \Big(
p \frac{\sigma^2}{2} \Delta f(x)\nonumber\\
 & \qquad \qquad \qquad +\big[\frac{x(x_0-x)}{2}-\big(d_0+\eta(x_0-x)\langle \bar{X}_s,1\rangle\big)\big]f(x)\Big)\bar{X}_s(dx)\, ds,\label{pbm_simu1}\\
& \langle M^f\rangle_t= \int_0^t \int_{\X}2 f^2(x)\bar{X}_s(dx)\, ds.\nonumber
\end{align}

\begin{figure}[!ht]
\begin{center}
\begin{tabular}{ccc}
  (a) &  \includegraphics[width=5cm,angle=0,trim= 0cm 0cm 0cm 0cm]{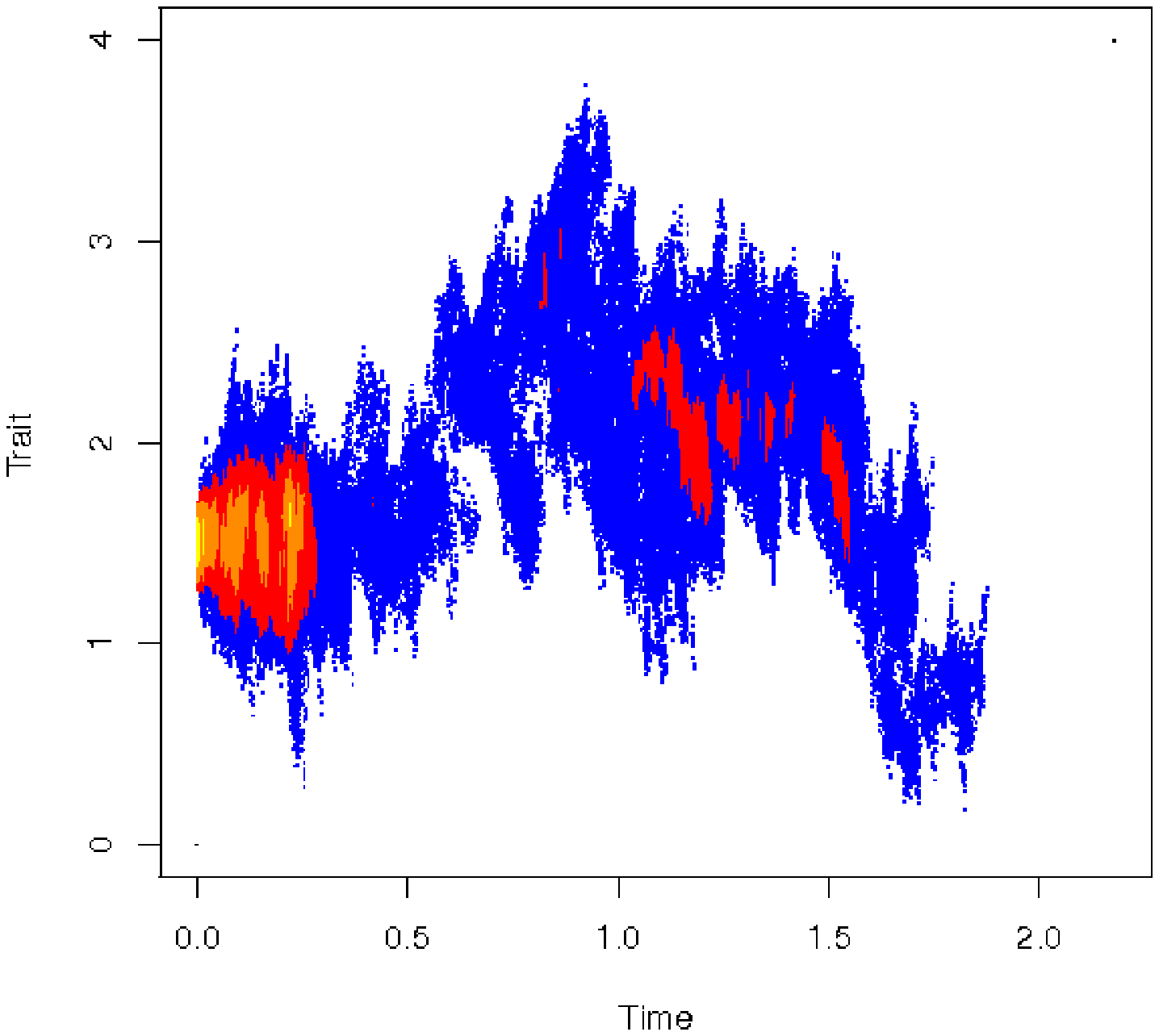} & \includegraphics[width=5cm,angle=0,trim= 0cm 0cm 0cm 0cm]{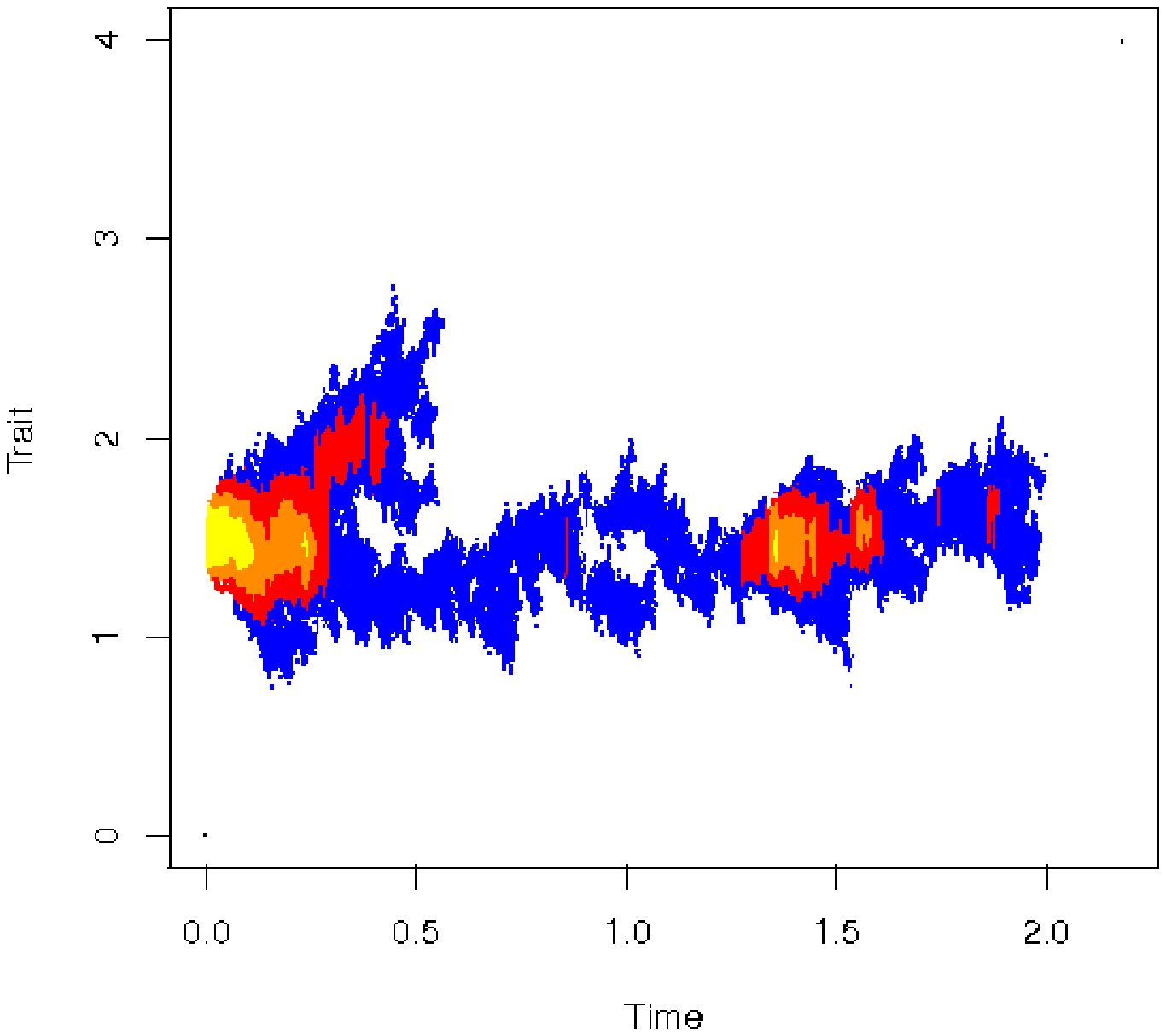}\\
(b) & \includegraphics[width=5cm,angle=0,trim= 0cm 0cm 0cm 0cm]{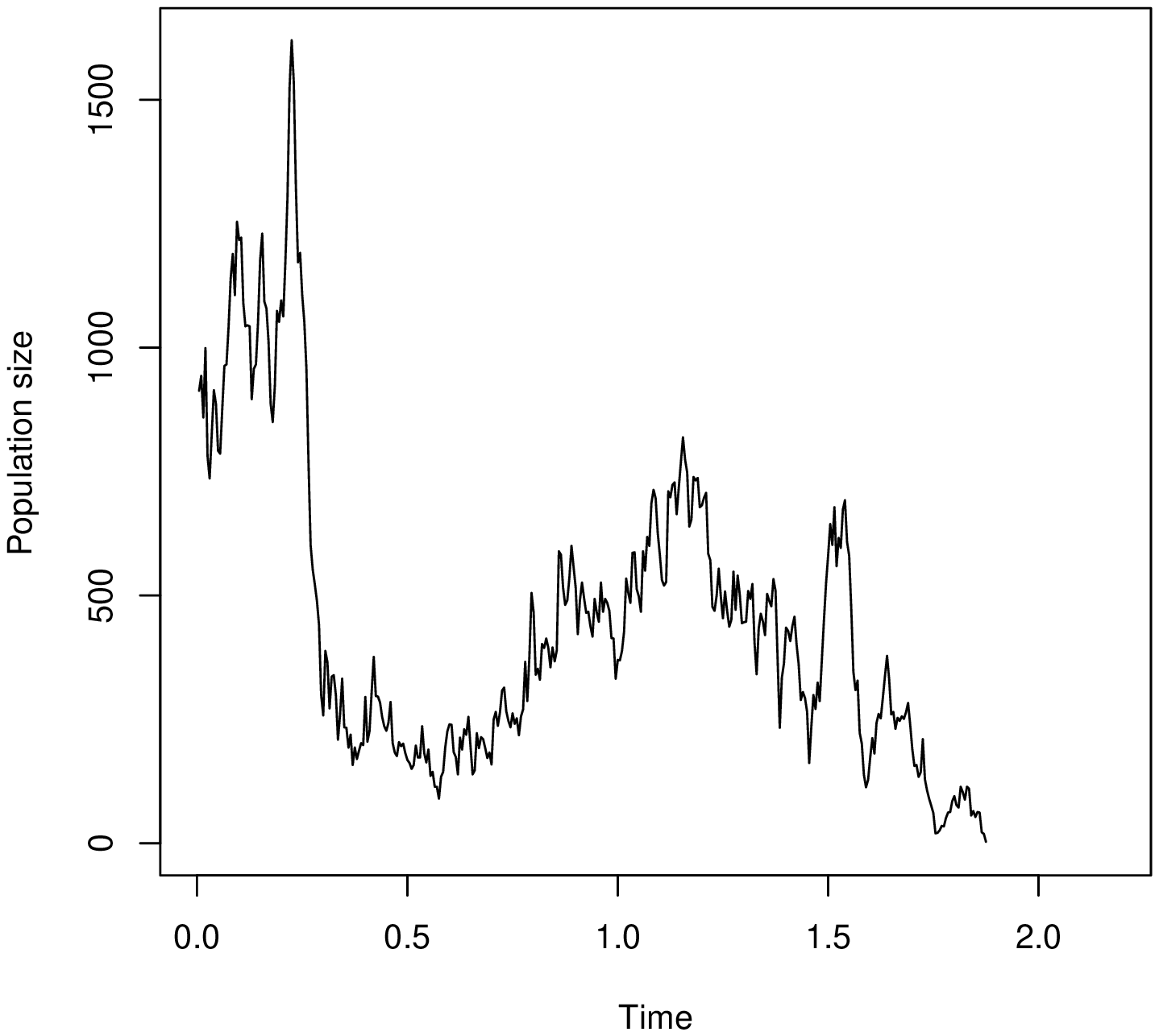} & \includegraphics[width=5cm,angle=0,trim= 0cm 0cm 0cm 0cm]{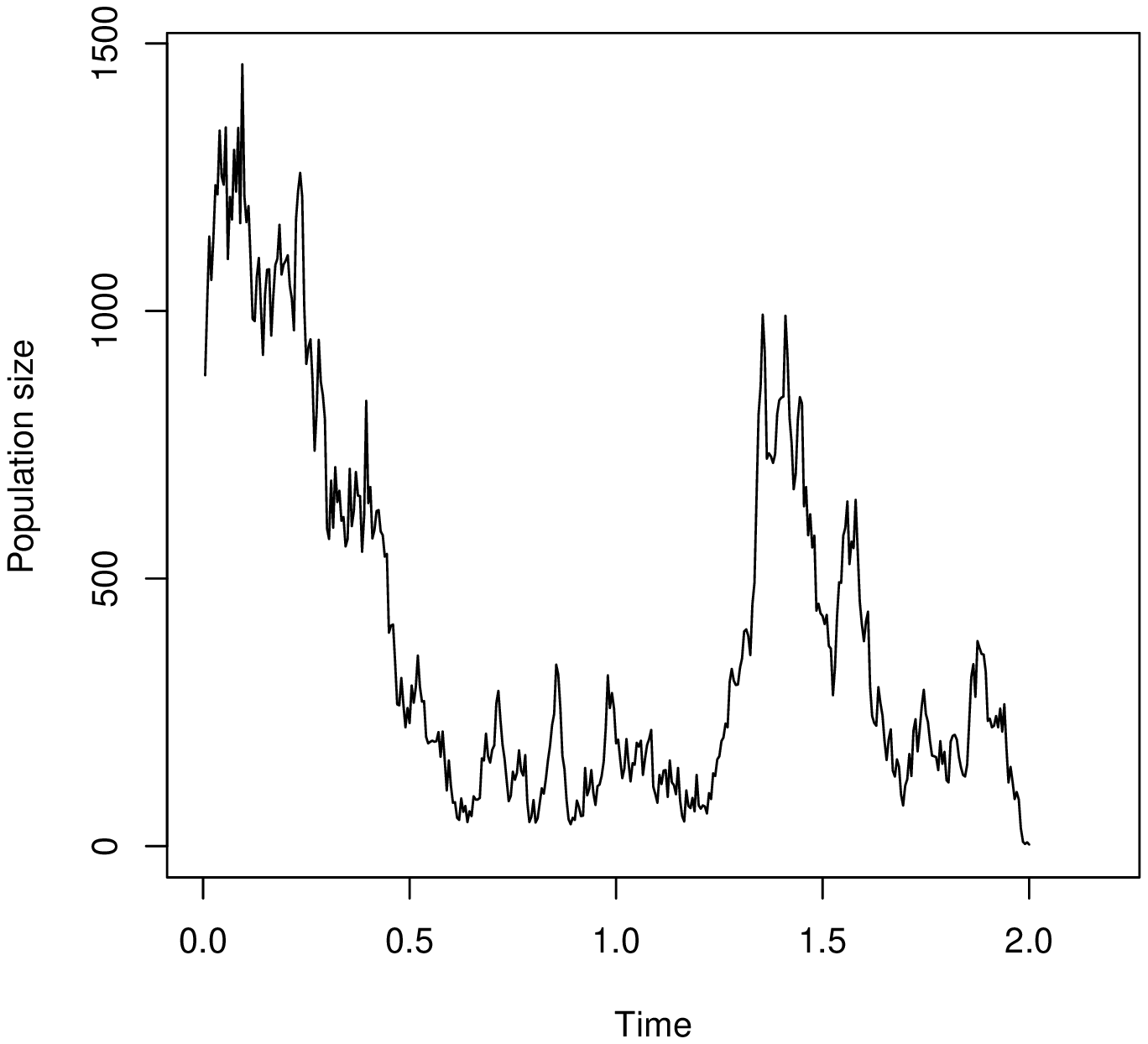} \\
(c) & \includegraphics[width=5cm,angle=0,trim= 0cm 0cm 0cm 0cm]{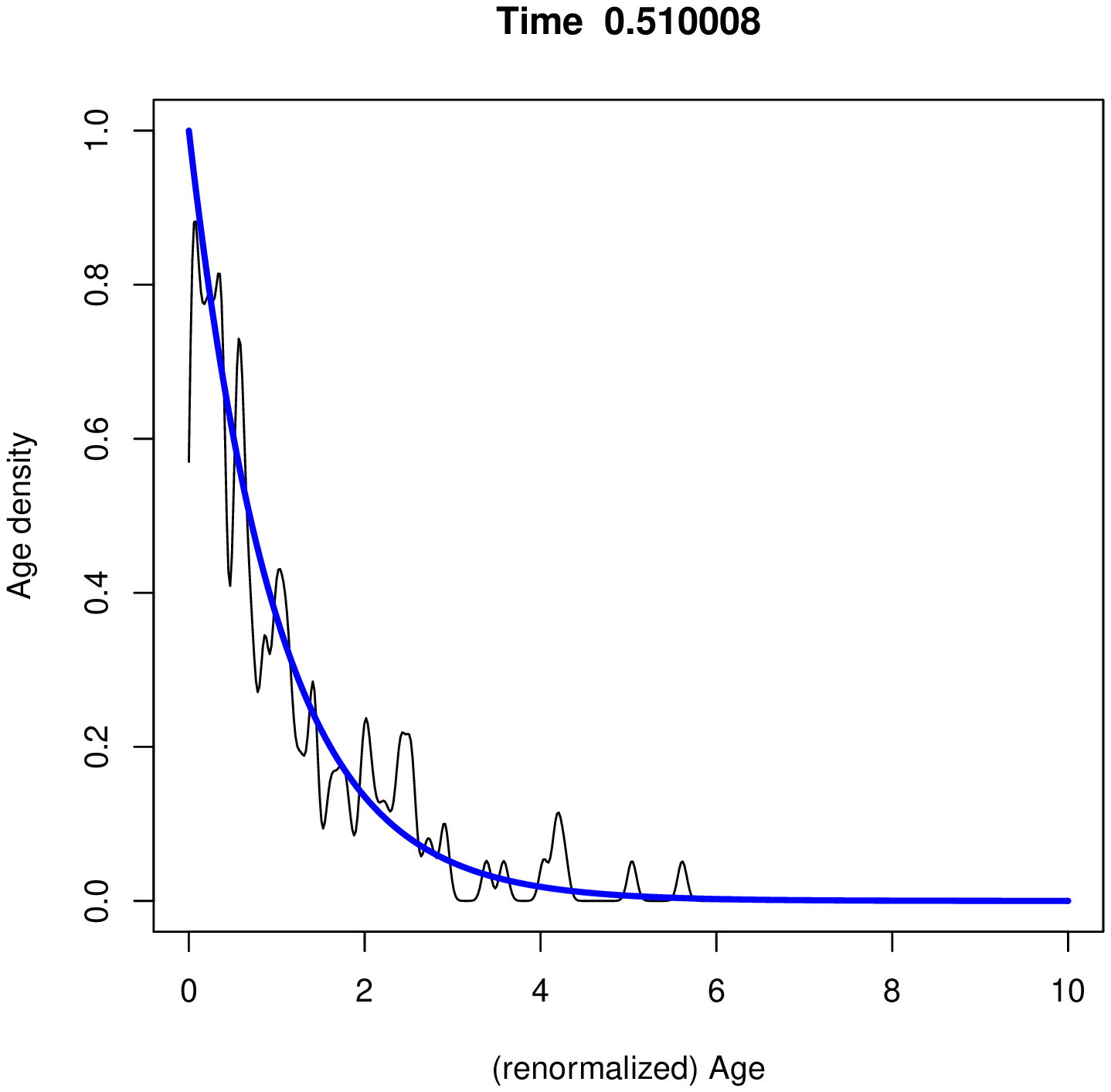} & \includegraphics[width=5cm,angle=0,trim= 0cm 0cm 0cm 0cm]{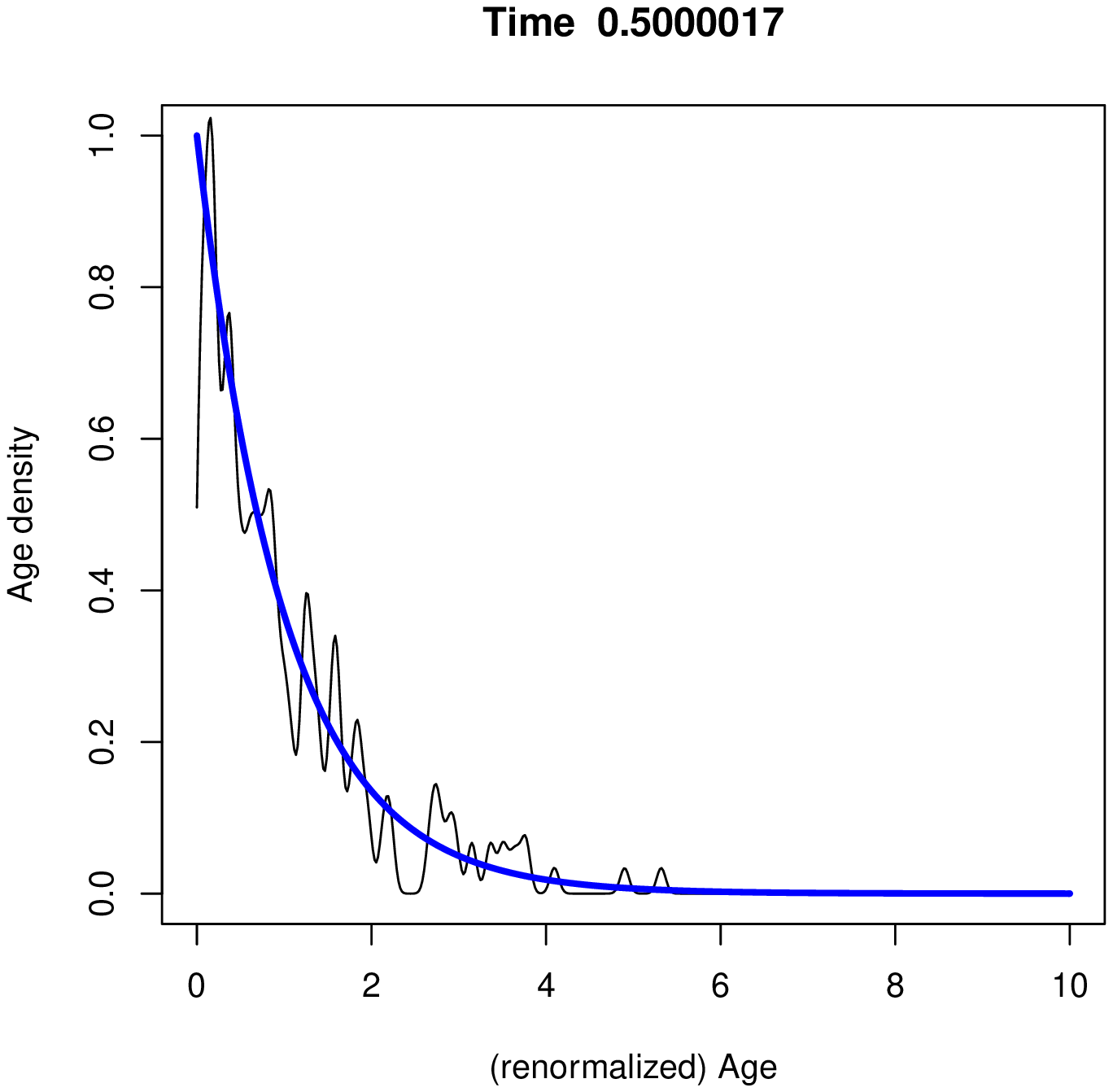}
\end{tabular}
\caption{\textit{Simulation of the individual-based process $X^n$, with $n=1000$ and discretization step $\Delta t=0.005$. The system is started with $1000$ particles of trait $x=1.5$. First line: $\sigma=1$. Second line: $\sigma=0.8$. (a): Support  of the process $\bar{X}^n$  (with time in abscissa and trait in ordinate).
(b): Evolution of the population size. (c): Age distribution for $t=0.5$. It can be checked that the age distribution converges to an exponential of parameter 1 (plain line).}}\label{simu1}
\end{center}
\end{figure}

\noindent In Figure  \ref{simu1}, two sets of simulations are presented, depending on two different mutation variances $\sigma^2$. As expected, when $\sigma$ increases, the traits vary more rapidly, and the irregularity of the trait support appears more strikingly. On both simulations of Fig. \ref{simu1}, extinction happens in a fast time. Almost-sure extinction is due to the logistic interaction, as proved in the following proposition.

\begin{prop}\label{prop_extinctionps}There is almost-sure extinction of the superprocess (\ref{pbm_simu1}).
\end{prop}

\begin{proof}
The mass of the super-process satisfies the following equation:
\begin{align}
 & \langle \bar{X}_t,1\rangle =    \langle \bar{X}_0,1\rangle
+\int_0^t \int_{[0,x_0]} \Lambda(x,\langle X_s,1\rangle) \bar{X}_s(dx)\,ds +M^1_t\nonumber\\
& \langle M^1\rangle_t= \int_0^t 2\langle \bar{X}_s,1\rangle\, ds,\nonumber\\
& \mbox{where }  \Lambda(x,Z)=\frac{x(x_0-x)}{2}-\big(d_0+\eta(x_0-x)Z\big).\label{masse_simu1}
\end{align}This equation is not closed for the mass process, since the drift depends on the trait distribution. Our purpose is to upper-bound $\Lambda(x,Z)$ so that $\langle \bar{X}_.,1\rangle$ can be stochastically dominated by a Feller diffusion with negative drift, that goes extinct almost surely.\\
In the case where $x>x_0-({2d_0\over x_0} -\zeta)$ with $\zeta\in (0,{2d_0\over x_0} \wedge 1)$ and since $x\in (0,x_{0})$, one gets
\begin{align}
\Lambda(x,Z)= & -d_0+(x_0-x)\Big(\frac{x}{2}-\eta Z\Big)\nonumber\\
\leq & -d_0+\Big(\frac{2d_0}{x_0}-\zeta\Big)\times \frac{x_0}{2}=-\frac{\zeta x_0}{2}.\label{pp1}
\end{align}
In the case where $x<x_0-({2d_0\over x_0}-\zeta)$, then $0<{2d_0\over x_0} -\zeta \leq x_0-x\leq x_0$ and depending on the sign of $x/2-\eta Z$:
\begin{align}
\Lambda(x,Z)\leq & -d_0+\max\Big(x_0\big(\frac{x_0}{2}-\eta Z\big)\ ; \
\big(\frac{2d_0}{x_0}-\zeta\big) \big(\frac{x_0}{2}- \eta Z\big)\Big)\nonumber\\
 \leq & \frac{x_0^2}{2}-d_0- \eta \big(\frac{2d_0}{x_0}-\zeta\big) Z.\label{pp2}
\end{align}
Since the upper bounds in (\ref{pp1}) and (\ref{pp2}) are equal when the mass $Z$ equals $m_{0}$ defined by
\begin{equation}
m_0=\frac{\frac{x_0(x_0+\zeta)}{2}-d_0}{\eta \big(\frac{2d_0}{x_0}-\zeta\big)},
\end{equation} we  thus get in any case that
\begin{align}
\Lambda(x,Z)\leq -\frac{\zeta x_0}{2}\ind_{Z \geq m_0}+\Big(\frac{x_0^2}{2}-d_0\Big)\ind_{Z \leq m_0}.\label{defBs}
\end{align}Hence, the process $\langle \bar{X}_.,1\rangle$ can be stochastically dominated by the following positive process:
\begin{equation}
Z_t=\langle \bar{X}_0,1\rangle+\int_0^t \Big(- \frac{\zeta x_0}{2} Z_s+ m_0\big(\frac{x_0(x_0+\zeta)}{2}-d_0\big)\ind_{Z_s\leq m_0}\Big)ds+\int_0^t \sqrt{2 Z_s}dB_s \label{processusZ}
\end{equation}where $B$ is a standard Brownian motion. \\
We can adapt the results of Meyn and Tweedie \cite{meyntweedie} to prove almost sure extinction. For $u\in \R_+$, let us denote by $\tau_{u}=\inf\{t\geq 0,\ Z_t\leq u\}$ and let $z=\langle \bar{X}_0,1\rangle>0$. Either $z\leq m_0$ and then $\tau_{m_0}=0$, or $z>m_0$. In the latter case, let us consider $M>z>m_{0}$ and let $\rho_M=\inf\{t\geq 0,\ Z_t\geq M\}$. We have
\begin{align}
  \E_z\Big(Z_{\tau_{m_0}\wedge \rho_M}-z+\frac{\zeta x_0}{2}\int_0^{\tau_{m_0}\wedge \rho_M}Z_s\, ds - \int_{0}^{\tau_{m_0}\wedge \rho_M} \sqrt{2 Z_s}dB_s \Big)=0.
 \end{align}
  By the uniform integrability of the fourth term and the optional stopping theorem, we have $ \E_z\Big( \int_{0}^{\tau_{m_0}\wedge \rho_M} \sqrt{2 Z_s}dB_s \Big)=0$. Since moreover $Z_{s}\geq m_{0}$ for $s\in [0,\tau_{m_0}\wedge \rho_M]$, we  deduce that
\begin{align}
\label{fin} m_0 \ \E_z\big(\tau_{m_0}\wedge \rho_M\big)\leq \E_z\Big(\int_0^{\tau_{m_0}\wedge \rho_M}Z_s\ ds \Big) \leq \frac{2z}{\zeta x_0}.
\end{align} It can easily be proved that for all $T>0$, $\mathbb{E}(\sup_{t\leq T} (Z_{t})^2)<\infty$, implying that $\rho_{M}$ tends to infinity with $M$. Thus, \eqref{fin} provides that for all $z>0$, $\P_z(\tau_{m_0}<+\infty)=1$. By Girsanov's theorem, there exists a probability measure under which the process $Z$ is a sub-critical Feller diffusion. It turns out that  $\P_{m_0}(\tau_0\wedge \rho_M<+\infty)=1$. Standard computation  using the strong Markov property yields $\P_z(\tau_0<+\infty)=1$.
\end{proof}

\subsection{Logistic biological-age and size-structured population}\label{sectionex2}

In this section, the trait $x\in [x_1,x_2]\subset (0,x_0)$ (with $x_1, x_2>0$) is linked to the rate of metabolism, which measures the energy expended by individuals, and is often an increasing function of the body size. Ageing may result from toxic by-products of the metabolism.  This leads us to introduce a biological age $xa$, where $a$ is the physical age and where $x$ can be interpreted as the ageing velocity equals. In this example, we consider $r(x,a)=xa$ so that biologically older individuals give birth and die with higher rate, the other parameters being chosen as in Subsection 4.1. For a review on body size, energy metabolism and ageing, we refer the reader to \cite{speakman}. In our example:
\begin{align}
\widehat{m}(x,a)=\frac{2\sqrt{x}e^{-\frac{xa^2}{2}}}{\sqrt{2\pi}}\ind_{[0,+\infty)}(a).
\end{align}We recognize the Gaussian distribution with variance $1/x$ conditioned on being positive. Then:
\begin{align}
\widehat{b}(x)= & \frac{2x^{3/2}(x_0-x)}{\sqrt{2\pi}}\int_0^{+\infty}e^{-\frac{x(a^2+2a/x)}{2}}da
=\frac{2x(x_0-x)e^{\frac{1}{2x}}}{\sqrt{2\pi}}\int_{1/\sqrt{x}}^{+\infty}e^{-\frac{\alpha^2}{2}}d\alpha\nonumber\\
= & 2x(x_0-x)e^{\frac{1}{2x}}\Phi\Big(-\frac{1}{\sqrt{x}}\Big),\\
\widehat{r}(x)= & \frac{2x^{3/2}}{\sqrt{2\pi}} \int_0^{+\infty} a e^{-\frac{xa^2}{2}}da
=\sqrt{\frac{2x}{\pi}}\int_0^{+\infty}\alpha e^{-\frac{\alpha^2}{2}}d\alpha=\sqrt{\frac{2x}{\pi}},
\end{align}where $\Phi$ is the cumulative distribution function of the standard Gaussian distribution. The functions $\widehat{d}(x)$ and $\widehat{U}(x,y)$ are unchanged and given by (\ref{bdUchapeau-exemple1}).
The martingale problem (\ref{martingalepremiercas}) becomes here:
\begin{align}
 & M^{f}_t=  \langle \bar{X}_t,f\rangle -\langle \bar{X}_0,f\rangle
-\int_0^t \int_{\X} \Big(
p \sqrt{\frac{x}{2\pi}}\sigma^2\Delta\varphi(x)+\big[2x(x_0-x)e^{\frac{1}{2x}}\Phi\Big(-\frac{1}{\sqrt{x}}\Big)\nonumber\\
& \qquad \qquad \qquad -\big(d_0+\eta(x_0-x)\langle \bar{X}_s,1\rangle\big)\big]f(x)\Big)\bar{X}_s(dx)\, ds \label{pbm_ex2}\\
& \langle M^f\rangle_t= \int_0^t \int_{\X}2\sqrt{\frac{2x}{\pi}}f^2(x)\bar{X}_s(dx)\, ds.\nonumber
\end{align}

\begin{figure}[!ht]
\begin{center}
\begin{tabular}{ccc}
  (a) & \includegraphics[width=5cm,angle=0,trim= 0cm 0cm 0cm 0cm]{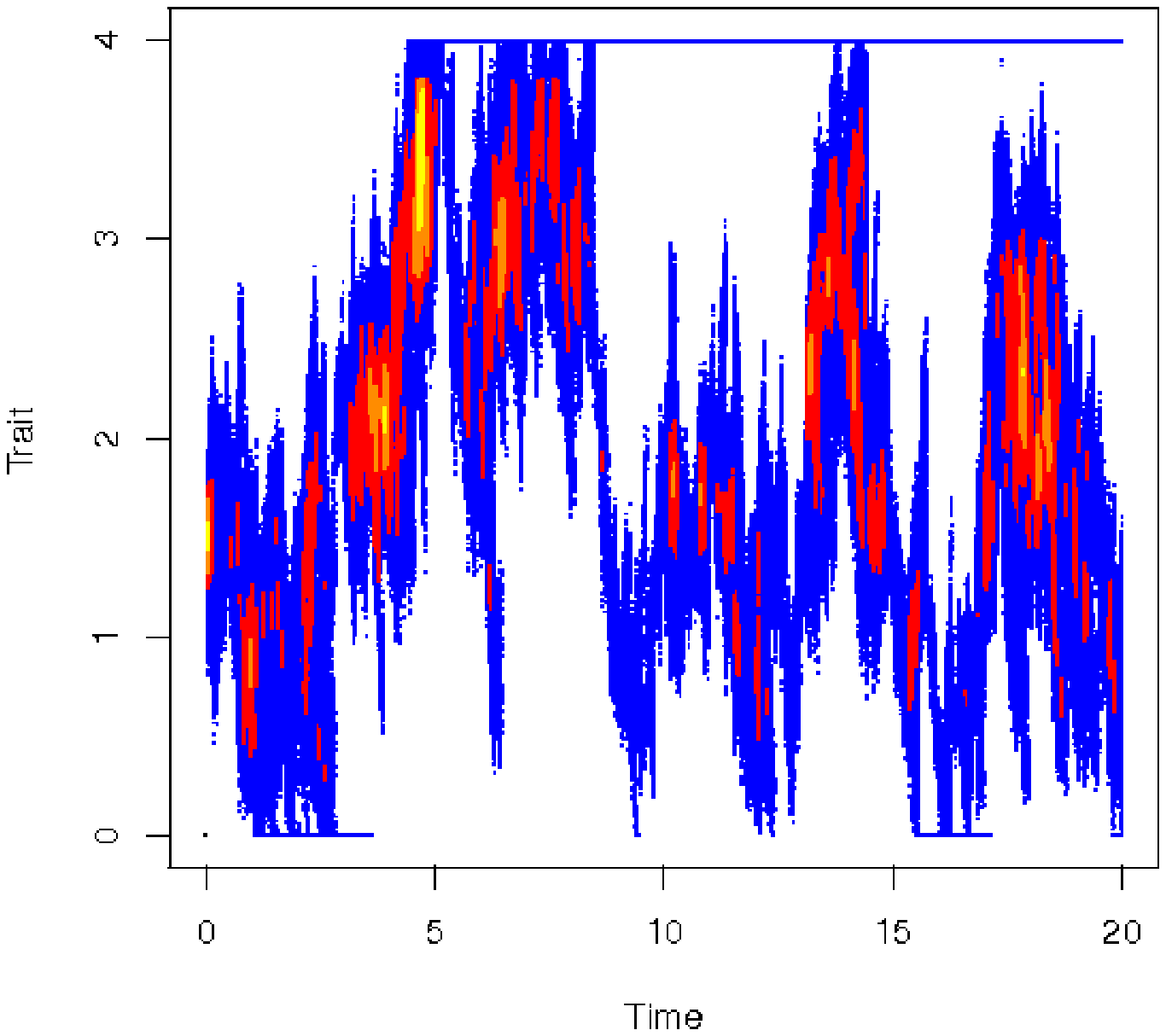} & \includegraphics[width=5cm,angle=0,trim= 0cm 0cm 0cm 0cm]{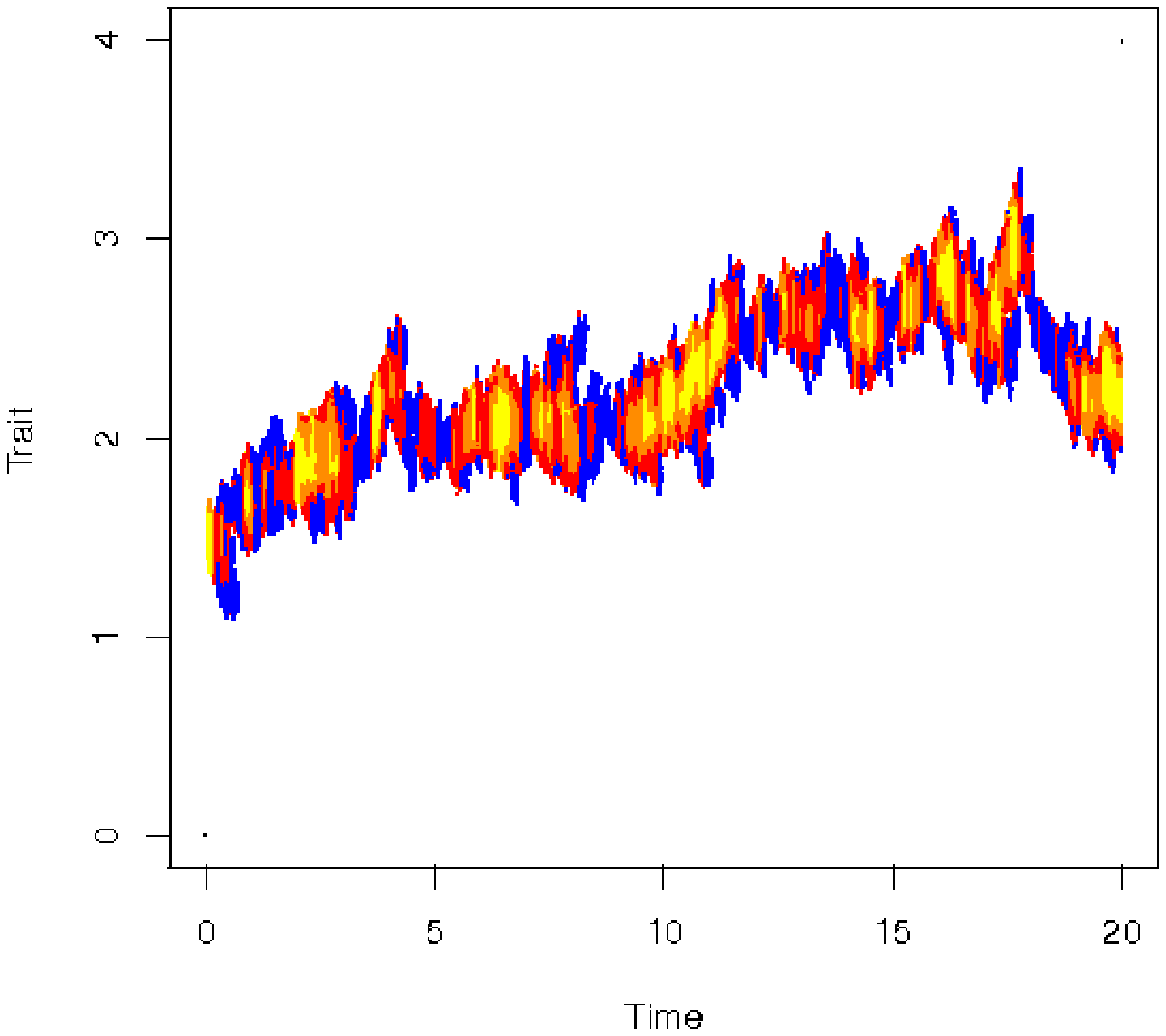}\\
 (b) & \includegraphics[width=5cm,angle=0,trim= 0cm 0cm 0cm 0cm]{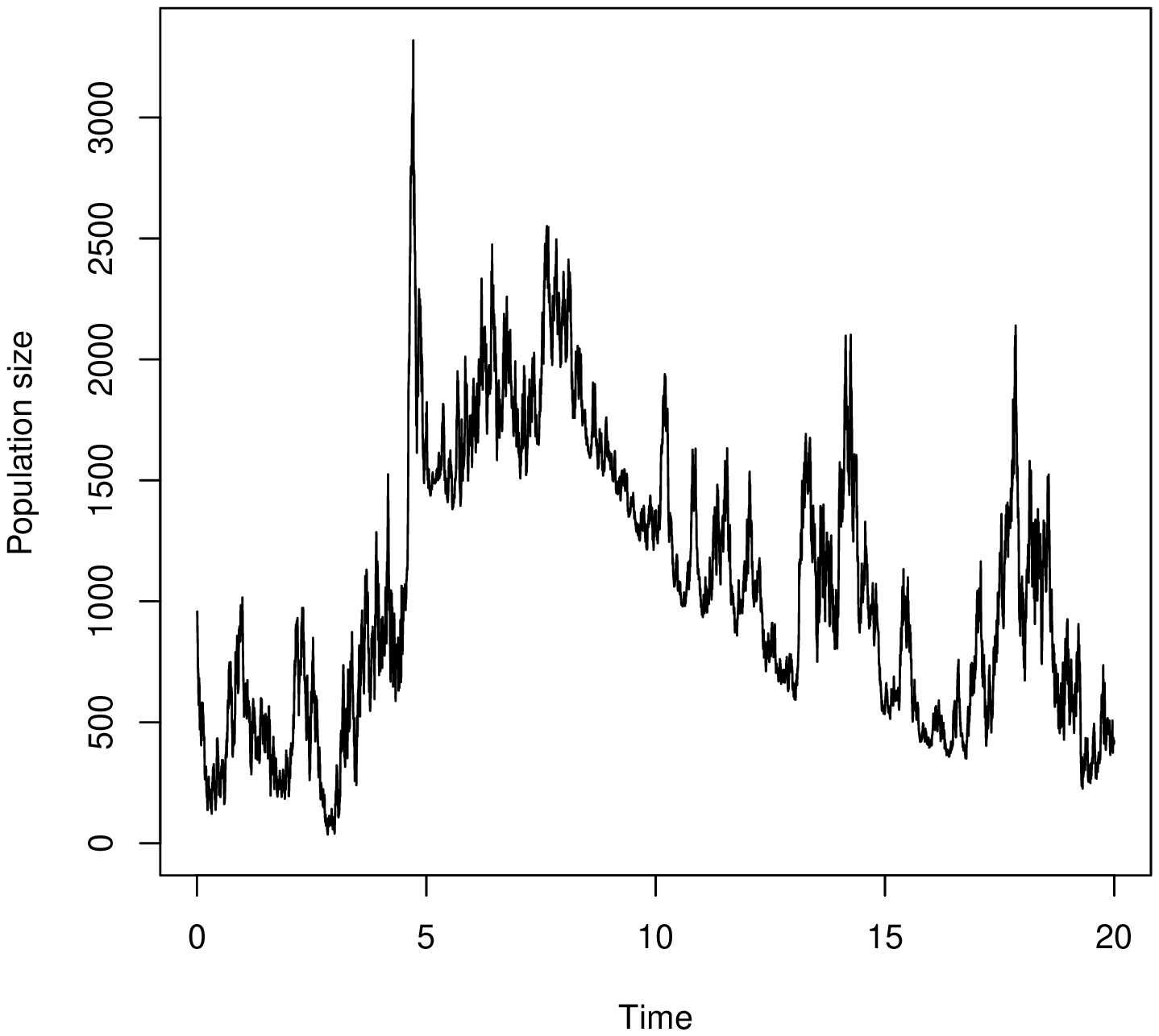} & \includegraphics[width=5cm,angle=0,trim= 0cm 0cm 0cm 0cm]{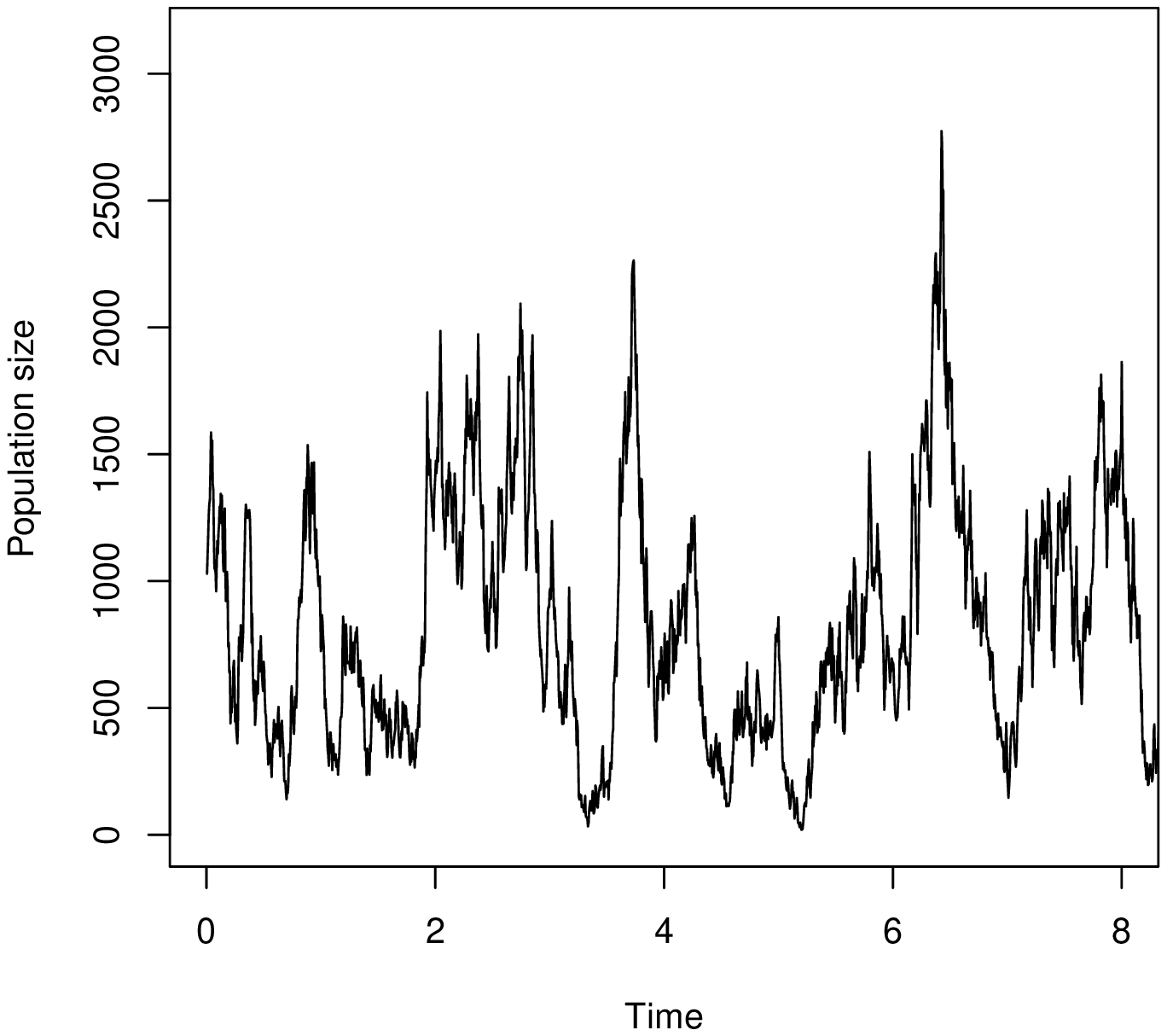}\\
(c) & \includegraphics[width=5cm,angle=0,trim= 0cm 0cm 0cm 0cm]{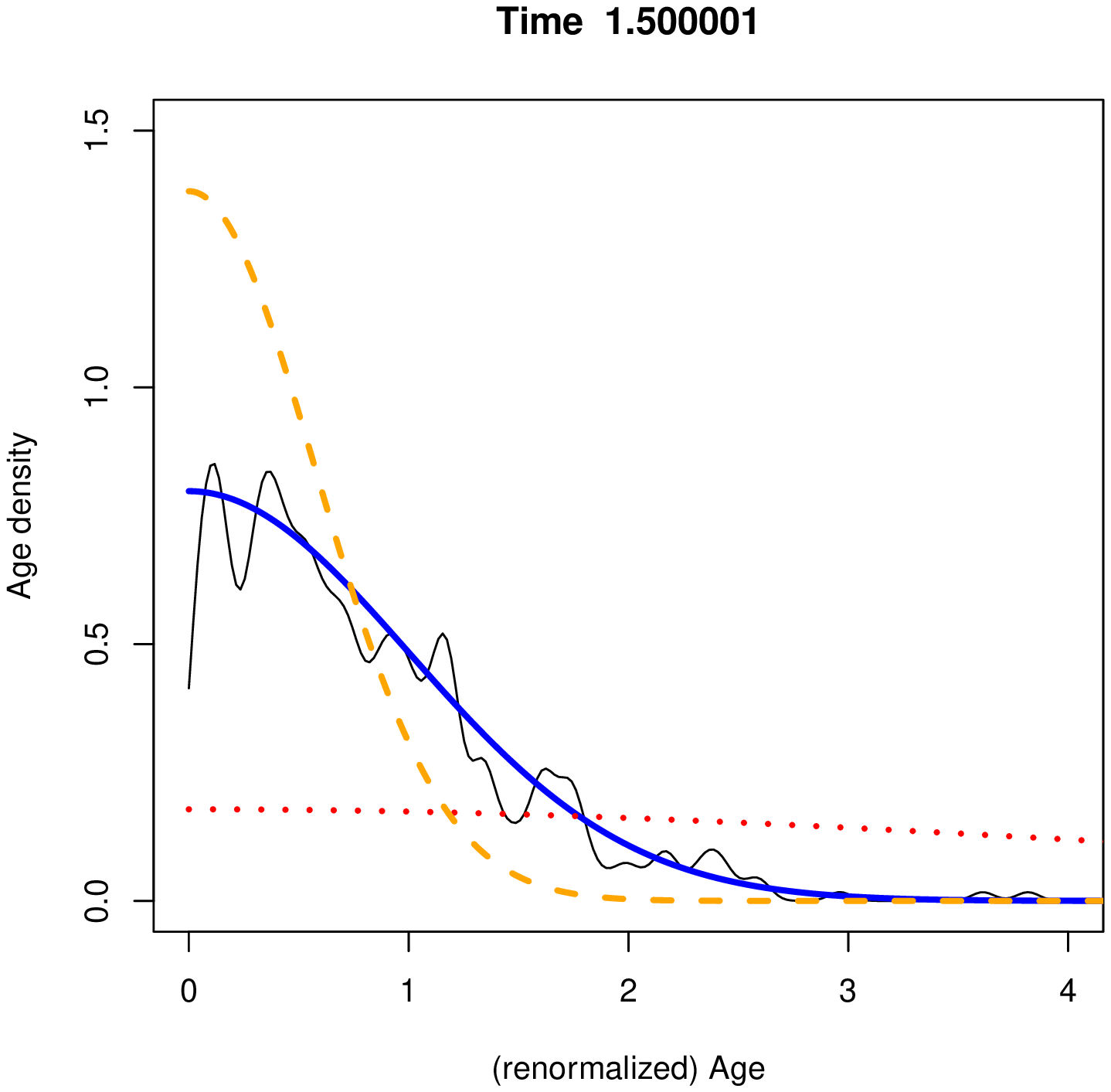} & \includegraphics[width=5cm,angle=0,trim= 0cm 0cm 0cm 0cm]{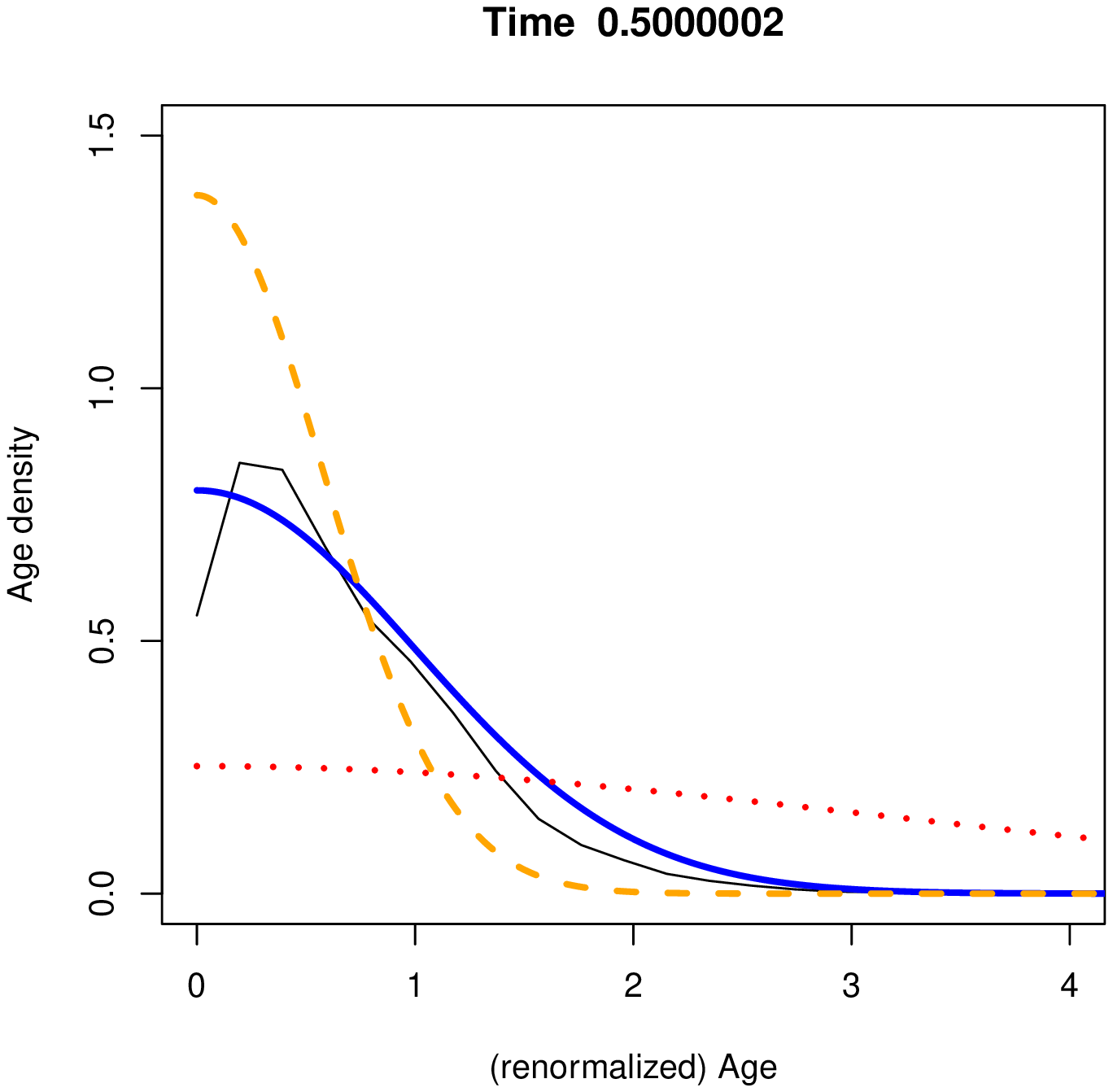}
\end{tabular}
\caption{\textit{Simulation of the individual-based process $X^n$, with $n=1000$ and discretization step $\Delta t=0.005$. The system is started with $1000$ particles of trait $x=1.5$. First line: $\sigma=1$, Second line: $\sigma=0.2$. (a): Support of the process $\bar{X}^n$  (with time in abscissa and trait in ordinate). (b): Evolution of the population size. (c): Marginal age distribution for $t=0.5$. For comparison, we draw the density $\widehat{m}(1,a)$ (plain line), $\widehat{m}(0.5,a)$ (dotted  line) and $\widehat{m}(3,a)$ (dashed  line).}}\label{simu2}
\end{center}
\end{figure}

\begin{figure}[!ht]
\begin{center}
\begin{tabular}{ccc}
  (a) & (b)\\
 \includegraphics[width=5cm,angle=0,trim= 0cm 0cm 0cm 0cm]{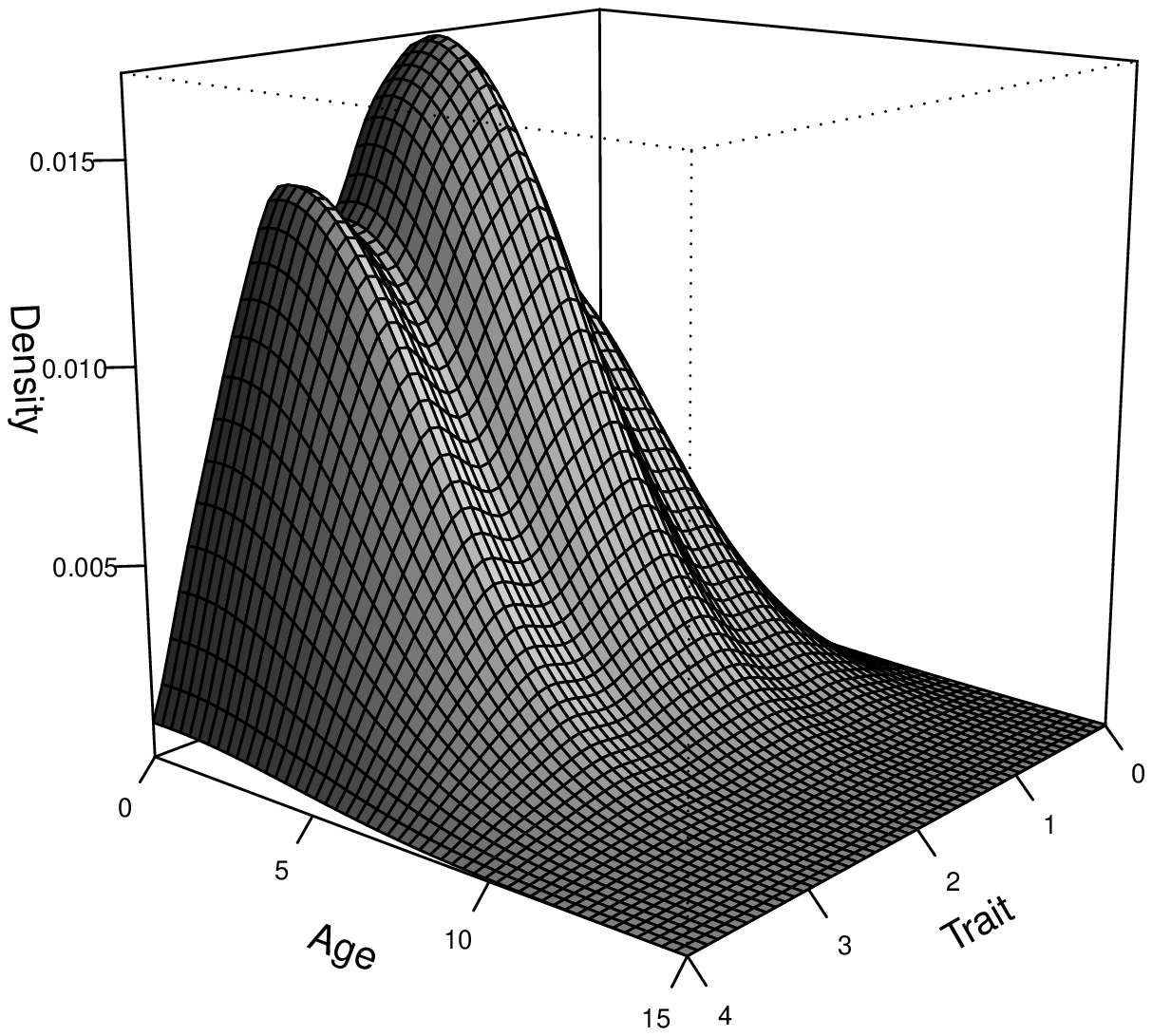} & \includegraphics[width=5cm,angle=0,trim= 0cm 0cm 0cm 0cm]{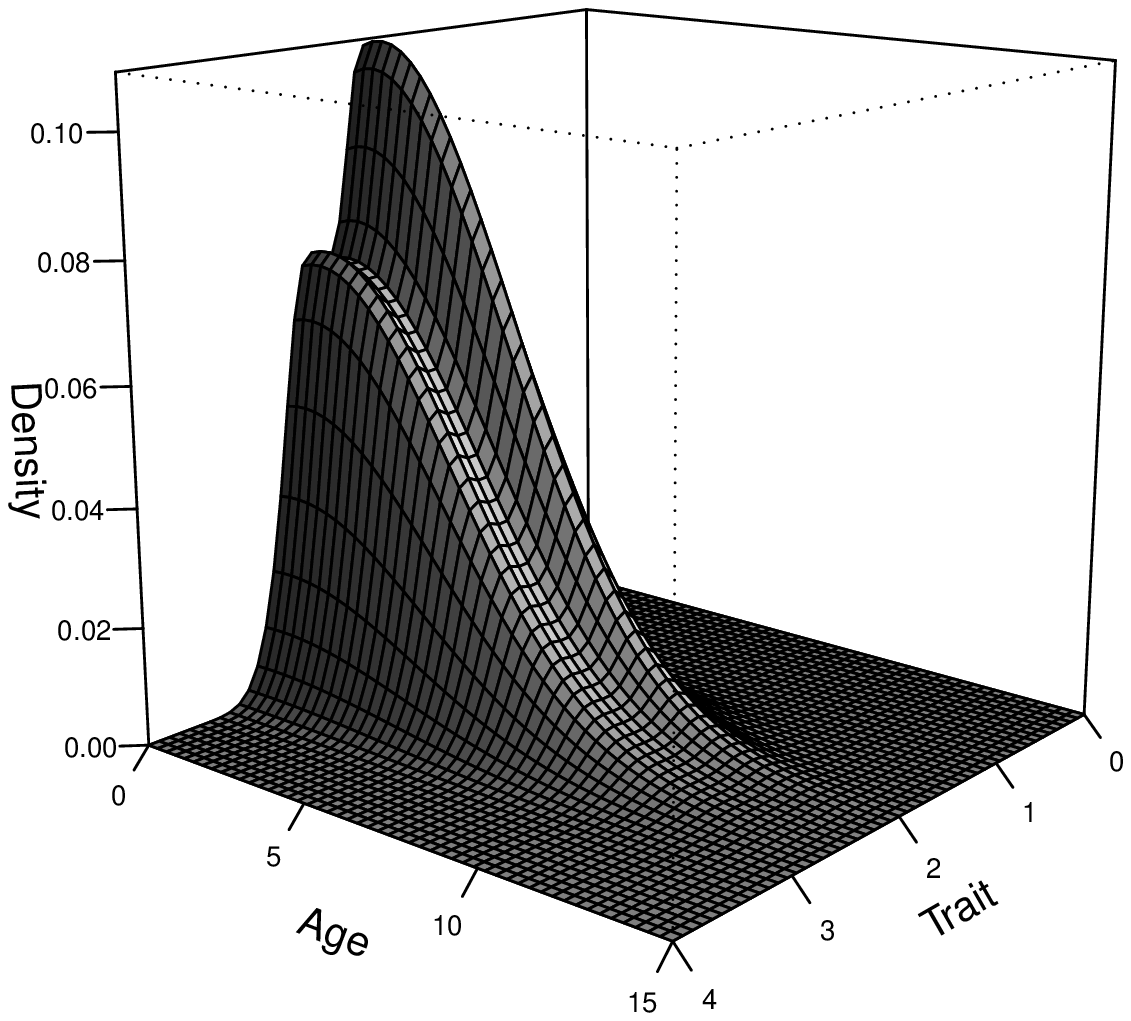}
\end{tabular}
\caption{\textit{Joint distributions in trait and age for the simulation of Fig. \ref{simu2}. We see that contrarily to the example of Section \ref{exemple_r(x)=1}, we do not have here independence of traits and age. (a): $\sigma=1$. (b): $\sigma=0.2$}}\label{simu2age}
\end{center}
\end{figure}

\me In this example, there is a higher senescence for individuals with trait $x>1$, compared with the example of Section \ref{exemple_r(x)=1}. The new choice of $r(x,a)$ influences the age distribution: lifelengths are shortened. This can be seen on the smaller support of the age distribution (compare Fig. \ref{simu2}-(c) with Fig. \ref{simu1}-(c)). \\
However, the populations are more persistent in the example of this section, although it can be proved similarly to Prop. \ref{prop_extinctionps} that there is almost sure extinction. Indeed, contrary to the populations in Example 1 which are extinct at $t=2$, the population of Example 2 still survives at $t=20$. One reason is that the growth rate in the finite variation term of \eqref{pbm_ex2} is bigger than the one in Example 1. Indeed, for many values of $x$, the factor $2\exp(1/2x)\Phi(-1/\sqrt{x})$ in the birth rate $\widehat{b}(x)$ is  bigger than the factor $1/2$. For $x=1.5$, $2\exp(1/2x)\Phi(-1/\sqrt{x})=0.58>0.5$ and for $x=3$, $2\exp(1/2x)\Phi(-1/\sqrt{x})=0.67>0.5$. \\
When comparing Fig. \ref{simu1}-(b) and Fig. \ref{simu2}-(b), we observe more fluctuations of the population size in Example 2. The bracket of the martingale in (\ref{pbm_ex2}) presents  a multiplicative $x$ term, compared to (\ref{pbm_simu1}). As soon as $x>{\pi\over 2}$, this explains the increased variance. Notice however that this variance tends to zero when the population size tends to zero, which also explains why there is no decrease in the population persistence.\\
Finally, the   multiplicative
 term  $p \sqrt{x/2\pi}\sigma^2$ in front of the diffusion term $\Delta\varphi(x)$ explains the large variability of the trait support, which  is observed  in Fig. \ref{simu2}-(a). When the diffusion coefficient $\sigma$ is small (second line of Fig. \ref{simu2}),  the traits evolve towards a value between $x=2$ and $x=4$ where the trade-off between reproduction and competition is optimized.\\


\noindent \textbf{Acknowledgements:} This work benefited from the support of the  ANR MANEGE (ANR-09-BLAN-0215) and from the "Chaire Modélisation Mathématique et Biodiversité of Veolia Environnement-Ecole Polytechnique-Museum National d'Histoire Naturelle-Fondation X". The authors thank S. Billiard, T. Kurtz and L. Popovic for enriching discussions.


\begin{thebibliography}{10}

\bibitem{athreyaathreyaiyer}
K.~Athreya, S.~Athreya, and S.~Iyer.
\newblock Super-critical age dependent branching {M}arkov processes and their
  scaling limits.
\newblock {\em Bernoulli}, 17(1):138--154, 2011.

\bibitem{ballkurtzpopovicrempala}
K.~Ball, T.G. Kurtz, L.~Popovic, and G.~Rempala.
\newblock Asymptotic analysis of multiscale approximations to reaction
  networks.
\newblock {\em Annals of Applied Probability}, 16(4):1925--1961, 2006.

\bibitem{bosekaj}
A.~Bose and I.~Kaj.
\newblock Diffusion approximation for an age-structured population.
\newblock {\em Annals of Applied Probabilities}, 5(1):140--157, 1995.

\bibitem{bosekaj2}
A.~Bose and I.~Kaj.
\newblock A scaling limit process for the age-reproduction structure in a
  {M}arkov population.
\newblock {\em Markov Processes and Related Fields}, 6(3):397--428, 2000.

\bibitem{champagnatferrieremeleard}
N.~Champagnat, R.~Ferri\`{e}re, and S.~M\'{e}l\'{e}ard.
\newblock Unifying evolutionary dynamics: from individual stochastic processes
  to macroscopic models via timescale separation.
\newblock {\em Theoretical Population Biology}, 69:297--321, 2006.

\bibitem{champagnatferrieremeleard2}
N.~Champagnat, R.~Ferri\`{e}re, and S.~M\'{e}l\'{e}ard.
\newblock From individual stochastic processes to macroscopic models in
  adaptive evolution.
\newblock {\em Stochastic Models}, 24:2--44, 2008.

\bibitem{charlesworth}
B.~Charlesworth.
\newblock {\em Evolution in Age structured Population}.
\newblock Cambridge University Press, 2 edition, 1994.

\bibitem{dawson}
D.~A. Dawson.
\newblock Mesure-valued markov processes.
\newblock In Springer, editor, {\em Ecole d'Et\'{e} de probabilit\'{e}s de
  Saint-Flour XXI}, volume 1541 of {\em Lectures Notes in Math.}, pages 1--260,
  New York, 1993.

\bibitem{dawsongorostizali}
D.A. Dawson, L.G. Gorostiza, and Z.~Li.
\newblock Nonlocal branching superprocesses and some related models.
\newblock {\em Acta Applicandae Mathematicae}, 74:93--112, 2002.

\bibitem{dynkin91}
E.B. Dynkin.
\newblock Branching particle systems and superprocesses.
\newblock {\em Annals of Probability}, 19:1157--1194, 1991.

\bibitem{etheridgebook}
A.~Etheridge.
\newblock {\em An introduction to superprocesses}, volume~20 of {\em University
  Lecture Series}.
\newblock Providence, {A}merican {M}athematical {S}ociety edition, 2000.

\bibitem{evansperkins}
S.N. Evans and E.A. Perkins.
\newblock Measure-valued branching diffusions with singular interactions.
\newblock {\em Canadian Journal of Mathematics}, 46:120--168, 1994.

\bibitem{evanssteinsaltz}
S.N. Evans and D.~Steinsaltz.
\newblock Damage segregation at fissioning may increase growth rates: A
  superprocess model.
\newblock {\em Theoretical Population Biology}, 71:473--490, 2007.

\bibitem{ferrieretran}
R.~Ferri\`ere and V.C. Tran.
\newblock Stochastic and deterministic models for age-structured populations
  with genetically variable traits.
\newblock {\em ESAIM: Proceedings}, 27:289--310, 2009.

\bibitem{fitzsimmons}
P.J. Fitzsimmons.
\newblock On the martingale problem for measure-valued markov branching
  processes.
\newblock In {\em Seminar on Stochastic Processes (Los Angeles, CA, 1991)},
  volume~29 of {\em Progr. Probab.}, pages 39--51, Boston, 1992. Birkha\"user.

\bibitem{fleischmannvatutinwakolbinger}
K.~Fleischmann, V.A. Vatutin, and A.~Wakolbinger.
\newblock Branching systems with long-living particles at the critical
  dimension.
\newblock {\em Theoretical Probability and its Applications}, 47(3):429--454,
  2002.

\bibitem{vonfoerster}
H.~Von Foerster.
\newblock Some remarks on changing populations.
\newblock In Grune~\& Stratton, editor, {\em The Kinetics of Cellular
  Proliferation}, pages 382--407, New York 1959.

\bibitem{jagersklebaner}
P.~Jagers and F.~Klebaner.
\newblock Population-size-dependent and age-dependent branching processes.
\newblock {\em Stochastic Processes and their Applications}, 87:235--254, 2000.

\bibitem{jagersklebaner2}
P.~Jagers and F.~Klebaner.
\newblock Population size dependent, age structured branching processes linger
  around their carrying capacity.
\newblock {\em Journal of Applied Probability}, page to appear, 2011.

\bibitem{joffemetivier}
A.~Joffe and M.~M\'{e}tivier.
\newblock Weak convergence of sequences of semimartingales with applications to
  multitype branching processes.
\newblock {\em Advances in Applied Probability}, 18:20--65, 1986.

\bibitem{JM09}
B.~Jourdain, S.~M\'{e}l\'{e}ard, and W.~Woyczynski.
\newblock L\'{e}vy flights in ecology.
\newblock preprint HAL-00560633.

\bibitem{kajsagitov}
I.~Kaj and S.~Sagitov.
\newblock Limit processes for age-dependent branching particle systems.
\newblock {\em Journal of Theoretical Probability}, 11(1):225--257, 1998.

\bibitem{kurtzaveraging}
T.G. Kurtz.
\newblock Averaging for martingale problems and stochastic approximation.
\newblock In Springer, editor, {\em Applied stochastic analysis (New Brunswick,
  NJ, 1991)}, volume 177 of {\em Lectures Notes in Control and Inform. Sci.},
  pages 186--209, Berlin, 1992.

\bibitem{mckendrick}
A.G. McKendrick.
\newblock Applications of mathematics to medical problems.
\newblock {\em Proc. Edin. Math.Soc.}, 54:98--130, 1926.

\bibitem{meleardroelly}
{\noopsort{Meleard}}{S. M\'{e}l\'{e}ard and S. Roelly}.
\newblock Sur les convergences \'{e}troite ou vague de processus \`{a} valeurs
  mesures.
\newblock {\em C.R.Acad.Sci.Paris, Serie I}, 317:785--788, 1993.

\bibitem{meleardtran}
{\noopsort{Meleard}}{S. M\'{e}l\'{e}ard and V.C. Tran}.
\newblock Trait substitution sequence process and canonical equation for
  age-structured populations.
\newblock {\em Journal of Mathematical Biology}, 58(6):881--921, 2009.

\bibitem{meyntweedie}
S.P. Meyn and R.L. Tweedie.
\newblock Stability of {M}arkovian processes {III}: {F}oster-{L}yapunov
  criteria for continuous-time processes.
\newblock {\em Advances in Applied Probability}, 25(3):518--548, 1993.

\bibitem{pazy}
A.~Pazy.
\newblock {\em Semigroups of Linear Operators and Applications to {P}artial
  {D}ifferential {E}quations}.
\newblock Springer-verlag edition, 1995.

\bibitem{rachev}
S.T. Rachev.
\newblock {\em Probability Metrics and the Stability of Stochastic Models}.
\newblock John Wiley \& Sons, 1991.

\bibitem{roelly}
S.~Roelly.
\newblock A criterion of convergence of measure-valued processes: Application
  to measure branching processes.
\newblock {\em Stochastics}, 17:43--65, 1986.

\bibitem{roellyrouault}
S.~Roelly and A.~Rouault.
\newblock Construction et propri\'{e}t\'{e}s de martingales des branchements
  spatiaux interactifs.
\newblock {\em International Statistical Review}, 58(2):173--189, 1990.

\bibitem{shorackwellner}
G.R. Shorack and J.A. Wellner.
\newblock {\em Empirical Processes with Applications to Statistics}.
\newblock Wiley, New-York, 1986.

\bibitem{speakman}
J.R. Speakman.
\newblock Body size, energy metabolism and lifespan.
\newblock {\em The Journal of Experimental Biology}, 208:1717--1730, 2005.

\bibitem{trangdesdev}
V.C. Tran.
\newblock Large population limit and time behaviour of a stochastic particle
  model describing an age-structured population.
\newblock {\em ESAIM: P\&S}, 12:345--386, 2008.

\bibitem{wangbio}
F.J.S. Wang.
\newblock A central limit theorem for age- and density-dependent population
  processes.
\newblock {\em Stochastic Processes and their Applications}, 5:173--193, 1977.

\bibitem{webb}
G.F. Webb.
\newblock {\em Theory of Nonlinear Age-Dependent Population Dynamics},
  volume~89 of {\em Monographs and Textbooks in Pure and Applied mathematics}.
\newblock Marcel Dekker, inc., New York - Basel, 1985.

\end{thebibliography}

{\footnotesize
\providecommand{\noopsort}[1]{}\providecommand{\noopsort}[1]{}\providecommand{\noopsort}[1]{}\providecommand{\noopsort}[1]{}

}\end{document}